\documentclass[a4paper]{amsart}

\usepackage{amsmath}
\usepackage{amsfonts}
\usepackage{amsthm}
\usepackage{graphicx}
\usepackage{eucal}
\usepackage{amscd}
\usepackage[all,2cell]{xy}
\usepackage{amssymb}
\usepackage{mathrsfs}

 \usepackage{tikz}
 \usetikzlibrary{positioning}
 \tikzset{mynode/.style={draw,circle,inner sep=1pt,outer sep=0pt}}

\newtheorem{teo}{Theorem}[section]
\newtheorem{cor}[teo]{Corollary}
\newtheorem{lem}[teo]{Lemma}
\newtheorem{defi}[teo]{Definition}

\newtheorem{prop}[teo]{Proposition}
\newtheorem{example}[teo]{Example}
\newtheorem{remark}[teo]{Remark}

\def\T{\mathbb{T}}

\newdir{ |>}{{}*!/-3.5pt/@{|}*!/-8pt/:(1,-.2)@^{>}*!/-8pt/:(1,+.2)@_{>}}

\dedicatory{}

\begin{document}

\title{ABSTRACTLY CONSTRUCTED PRIME SPECTRA}

\author{Alberto Facchini}
\address[Alberto Facchini]{Dipartimento di Matematica ``Tullio Levi-Civita'', Universit\`a di Padova, Via Trieste 63, 35121 Padova, Italy}
\thanks{}
\email{facchini@math.unipd.it}

\author{Carmelo Antonio Finocchiaro}
\address[Carmelo Antonio Finocchiano]{Dipartimento di Matematica e Informatica, Universit\`a\ di Catania, Citt\`a\ Universitaria, viale Andrea Doria 6, 95125 Catania, Italy}
\thanks{}
\email{cafinocchiaro@unict.it}

\author{George Janelidze}
\address[George Janelidze]{Department of Mathematics and Applied Mathematics, University of Cape Town, Rondebosch 7700, South Africa}
\thanks{}
\email{george.janelidze@uct.ac.za}

\subjclass[2010]{}

\date{\today}

\keywords{multiplicative lattice, complete lattice, algebraic lattice, lattice of ideals, lattice of congruences, sober space, spectral space, coherent space, prime spectrum, Zariski topology, compactness, prime element, radical element, solvable element, prime ideal, commutator}

\subjclass[2020]{06F99, 13A15, 14A05, 06D05, 16Y60, 20M12, 16D25, 18E13, 06D22, 54D30,	08A30, 08B99, 08B10, 16Y30}

\maketitle

\begin{abstract} The main purpose of this paper is a wide generalization of one of the results abstract algebraic geometry begins with, namely of the fact that the prime spectrum $\mathrm{Spec}(R)$ of a unital commutative ring $R$ is always a spectral (=coherent) topological space. In this generalization, which includes several other known ones, the role of ideals of $R$ is played by elements of an abstract complete lattice $L$ equipped with a binary multiplication with $xy\leqslant x\wedge y$ for all $x,y\in L$. In fact when no further conditions on $L$ are required, the resulting space can be and is only shown to be sober, and we discuss further conditions sufficient to make it spectral. This discussion involves establishing various comparison theorems on so-called prime, radical, solvable, and locally solvable elements of $L$; we also make short additional remarks on semiprime elements. We consider categorical and universal-algebraic applications involving general theory of commutators, and an application to ideals in what we call the commutative world. The cases of groups and of non-commutative rings are briefly considered separately.

\end{abstract}

\section{Introduction}

What is a spectral space, also called a coherent topological space? There is a purely topological definition, recalled in the next section, but, even more importantly in a sense, it has two well-known algebraic origins, which can be used as algebraic definitions:

\textbf{Fact 1.1.} The prime spectrum $\mathrm{Spec}(R)$ of a unital commutative ring $R$, defined as the set of prime ideals of $R$ equipped with Zariski topology, is a spectral space. Moreover, as shown by M. Hochster \cite{[H1969]}, every spectral space occurs this way.

\textbf{Fact 1.2.} The opposite category of distributive lattices is equivalent to the category of spectral spaces (with so-called spectral maps as morphisms). This is clearly explained e.g. by P. T. Johnstone in \cite{[J1982]}; like in that book, we assume all lattices to be bounded.

Obtaining the same kind of spaces from unital commutative rings and from distributive lattices could look surprising before having the following far more recent result of A. Pe\~na, L. M. Ruza, and J. Vielma \cite{[PRV2009]}:

\textbf{Fact 1.3.} Fact 1.1 extends to unital commutative semirings as Theorem 3.1 of \cite{[PRV2009]}. In particular, this theorem applies to unital commutative rings and to distributive lattices at the same time.

However, this result begs for a further generalization, as the following facts suggest:

\textbf{Fact 1.4.} Once commutative rings are replaced with commutative semirings, why not removing the additive structure completely? And indeed, the commutative monoid counterpart of Fact 1.3 (and more) can be found in \cite{[V2010]}. 

\textbf{Fact 1.5.} Let $R$ be a non-commutative unital ring. A prime ideal of $R$ can still be defined as a proper ideal $P$ with $XY\subseteq P\Rightarrow(X\subseteq P\,\,\,\text{or}\,\,\,Y\subseteq P)$ for all ideals $X$ and $Y$ of $R$, and we can still consider $\mathrm{Spec}(R)$ defined as the set of prime ideals of $R$ equipped with the Zariski topology. This space is not spectral in general, but:
\begin{itemize}
	\item As shown by I. Kaplansky \cite{[K1974]} it is spectral if the product of any two finitely generated ideals of $R$ is finitely generated itself; he called such $R$ a neo-commutative ring.
	\item L. P. Belluce \cite{[B1991]} proved the same assuming that the radical of the product of any two principal ideals of $R$ is the radical of a finitely generated ideal; he called such $R$ quasi-commutative.
	\item However, it was I. Kaplansky again who gave an example of a quasi-commu-tative ring that is not neo-commutative \cite{[K1994]}. For some further remarks see I. Klep and M. Tressl \cite{[KT2007]}.  
\end{itemize}

\textbf{Fact 1.6.} Discussing analogies between groups and commutative rings, where commutators of normal subgroups play the role of products of ideals, E. Schenkman \cite{[S1958]} (see also K. K. Shchukin \cite{[S1960]}) calls a normal subgroup $P$ of a group $G$ prime, if $[X,Y]\subseteq P\Rightarrow(X\subseteq P\,\,\,\text{or}\,\,\,Y\subseteq P)$ for all normal subgroups $X$ and $Y$ of $G$. However, general commutator theory tells us that this is more than just an analogy (discussing which we don't require rings to be unital and identify ring congruences with ideals). In fact there are more than three general notions of commutator that give the product of ideals in the case of (ideals of) commutative rings and the commutator of normal subgroups in the case of (normal subgroups of) groups. The main three such commutators are P. Higgins' commutator \cite{[H1956]}, S. A. Huq's commutator \cite{[H1968]}, and J. D. H. Smith's commutator \cite{[S1976]}. It is interesting that in the case of ideals of non-commutative rings each of three above-mentioned commutators of $X$ and $Y$ is equal to $XY+YX$, not necessarily to $XY$. Therefore commutator theory suggests to define a prime ideal via the implication $XY+YX\subseteq P\Rightarrow(X\subseteq P\,\,\,\text{or}\,\,\,Y\subseteq P)$. Fortunately, this definition, seemingly strange from the viewpoint of ring theory, is equivalent to the usual one.

\textbf{Fact 1.7.} The so-called ``abstract theory of ideals'' has almost hundred years of development. It replaces ideals of a ring, commutative or not, with elements of an abstract complete lattice $L$ with a binary multiplication replacing multiplication of ideals. The multiplication sometimes satisfies certain conditions; for example:
\begin{itemize}
	\item W. Krull \cite{[K1924]}, speaking of ``axiomatic theory of ideals'', in fact implicitly introduces certain structures involving order, multiplication and division, where the axioms involving multiplication and lattice operations are: associativity and commutativity of multiplication, distributivity of multiplication with respect to arbitrary non-empty joins, and an inequality, which in the present day language would be written as $ab\leqslant a\wedge b$ (for all $a$ and $b$), while Krull writes $\mathfrak{a}\cdot\mathfrak{b}\geq[\mathfrak{a},\mathfrak{b}]$. Of course $\leqslant$ versus $\geq$ is nothing but imitation of ideal inclusion versus imitation of inequalities between natural numbers that generate ideals of the ring of integers. Next, in \cite{[K1928]} (where he writes the inequality above as $\mathfrak{a}\cdot\mathfrak{b}\geq\mathfrak{a}\cap\mathfrak{b}$), he does not require the commutativity of multiplication anymore, thinking of ideals of non-commutative rings. In both papers `abstract' prime ideals are involved.  
	\item The book \cite{[D1990]}, dedicated to R. P. Dilworth's and related work, has a chapter called ``Multiplicative lattices'' with several old reprinted and new papers, where lattices equipped with a multiplication and/or residuation (=`division') are explicitly introduced and studied; the earlier papers \cite{[W1937]} and \cite{[W1938]} of M. Ward should also be mentioned here. Note also that Dilworth's introduction (called ``Background'') to ``Multiplicative lattices'' begins with ``In the middle 1930's Morgan Ward, largely from reading the work of E. Noether,	became convinced that much of the basic structure theory of commutative rings could be formulated in lattice theoretic terms provided an appropriate multiplication was defined over the lattice.'' The axioms involving multiplication and lattice operations have only minor differences from those of Krull. Note: we omit mentioning more recent related literature that does not follow the direction we are interested in this paper. References to that literature (of 20th century) can be found in an important survey \cite{[AJ2001]} of D. D. Anderson and E. W. Johnson, which mentions many of their papers among others.
	\item S. A. Amitsur \cite{[A1952]} goes much further and developes a general theory of radicals in an (abstract) lattice equipped with what he calls an $H$-relation. For him a lattice equipped with a multiplication is only a very special case, and, moreover, his multiplication is neither associative nor commutative in general; he does require, however, the distributivity of multiplication with respect to binary joins and the inequality $ab\leqslant a\wedge b$. The reason for abandoning associativity and commutativity is the desire to consider ideals of non-associative (non-commutative) rings, which he does in \cite{[A1954a]} and in \cite{[A1954b]}, although the case of associative rings remains the most important one there. He also indicates, in \cite{[A1954a]}, that many properties of radicals extend from ideals of rings to ideals in categories satisfying suitable conditions, saying: ``We do not intend to give here the list of axioms such a category has to satisfy, but only a list of conditions, some of which may be considered as axioms, others as lemmas, to be valid in such a category in order that the whole theory can be developed in it''.
	\item The expression ``abstract ideal theory'' used by various authors, appears, in particular, in the book \cite{[B1973]} of Garrett Birkhoff, where he begins Section 10 of Chapter XIV with ``The general theory of ideals in Noetherian rings centers around the concepts of primary and irreducible ideal, and of the radical of an ideal. As was first shown	by Ward and Dilworth \cite{[WD1939]}\footnote{This paper is one of those reprinted in \cite{[D1990]}.}, Part IV (here and below we are using our numeration of references), much of this theory is true in general Noetherian $l$-monoids. The present section develops this idea'', and adds in the footnote ``This idea was implicit in W. Krull \cite{[K1924]}, and developed by him in many later papers''. A few lines below, taking about radicals, he adds another footnote saying: ``For other lattice-theoretic approaches to the radical, see \cite{[A1952]}-\cite{[A1954b]} and \cite{[B1956]}''. What Birkhoff calls ``$l$-monoid'' is the same as a lattice equipped with a multiplication that is associative, distributive with respect to binary joins, and has an identity element. But Birkhoff considers several other structures including $m$-lattices (=$l$-groupoids), where associativity and the existence of identity element is not required. He also considers ``prime elements'' and compares them with ``maximal elements'' in rather general contexts.
	\item K. Keimel \cite{[K1972]} seems to be the first to generalize Zariski topology (which he, as well as some other authors, calls hull kernel topology), defining it on the set of minimal prime elements of a lattice equipped with a multiplication. He points out that his ``setting is close to that of Steinfeld \cite{[S1968]}'', although Steinfeld's multiplication is not necessarily binary, but it is $n$-ary, for an arbitrary $n\geqslant2$ (and no topology is involved).    	        
\end{itemize}   

\textbf{Fact 1.8.} A systematic study of prime ideals in categories satisfying carefully chosen axioms (in fact due to S. A. Huq \cite{[H1968]}) and in varieties of groups with multiple operators in the sense of P. J. Higgins \cite{[H1956]} was initiated by A. Buys and her collaborators and students, especially S. G. Botha and G. K. Gerber: see \cite{[BG1982]}, \cite{[G1983]}, \cite{[BG1985]}, \cite{[BB1985]}, \cite{[B1986]}, \cite{[G1989]}, and related papers. Independently of that, A. Ursini \cite{[U1984]} introduced and studied prime ideals in ideal determined varieties of universal algebras. The term ``ideal determined'' was first used by him and H.-P. Gumm in \cite{[GU1984]}; originally \cite{[U1972]} he used the term ``BIT'' (as an abbreviation of ``buona teoria degli ideali''). The precise relationship between categories satisfying Huq's axioms and ideal determined varieties of universal algebras was clarified much later, via semi-abelian categories in the sense of \cite{[JMT2002]} and ideal determined categories in the sense of \cite{[JMTU2010]}. At the same time, varieties of groups with multiple operators were the first motivating special case for ideal determined varieties of universal algebras. What seems to be most important to mention, in connection with what are doing here, are two papers of P. Aglian\`o, where the prime spectra of universal algebras were considered for the first time, for ideal determined varieties \cite{[A1989]} and then, much more generally, for congruence modular varieties \cite{[A1993]}. 

\textbf{Fact 1.9.} An extensive literature is devoted to quantales (see K. Rosenthal's book \cite{[R1990]} and its references), introduced by C. J. Mulvey \cite{[M1986]}, not to mention locales (see \cite{[J1982]} and its references). We will recall the definitions, and results on locales we use, in Section 3. Now let us only point out that many kinds of spectra of a quantale are considered in \cite{[R1990]}; see also \cite{[FFS2016a]} and \cite{[FFS2016b]} for a `non-quantale' approach to (different) spectra.\\   

The main purpose of this paper is to prove that $\mathrm{Spec}(L)$, the space of prime elements of a complete lattice $L$ equipped with a multiplication (we say ``complete multiplicative lattice''), is spectral in a clearly motivated general situation, and to show that this applies to categorical and universal-algebraic contexts really far more generally, than it was known before. Briefly, the paper is organized as follows:

In Section 2 we introduce our notion of complete multiplicative lattice (Definition 2.1), and denote, there and afterwards, a fixed such structure by $L$. The only requirement on the multiplication of $L$ is the inequality $xy\leqslant x\wedge y$ for all $x,y\in L$ (we use the letters $x$ and $y$ instead of $a$ and $b$). This so mild requirement is exactly what we need to prove that the space $\mathrm{Spec}(L)$ of prime elements of $L$, defined in the standard way, is a sober space. In particular, prime and radical elements are defined, and this is done in a standard way. 

In Section 3 we introduce `minimum' of relevant categories and functors and recall some `pointfree topology', essentially from \cite{[J1982]}, to present an equivalent condition for $\mathrm{Spec}(L)$ to be spectral in terms of the frame $\sqrt{L}$ of its radical elements (Theorem 3.6).

The purpose of Section 4 is to deepen Theorem 3.6 by using the radical closure operator on $L$ instead of using $\sqrt{L}$. The result is Theorem 4.4, which, unfortunately, gives only sufficient (not necessary) conditions for the spectrality of $\mathrm{Spec}(L)$. It is interesting that while its condition (d) seems to suggest that it covers Belluce's result mentioned in Fact 1.5, our further analysis (in Subsection 12.7) of its condition (c) shows that it only covers Kaplansky's result (see Fact 1.5 again).

Section 5 makes a straightforward generalization of known observations in order to show that $\mathrm{Spec}(L)$ is often `large enough to be interesting' unlike e.g. the special case where $xy=0$ for all $x,y\in L$, making $\mathrm{Spec}(L)$ empty. 

The notion of \textit{solvability} introduced in Section 6, contrary to \textit{nilpotency} used in ring theory for the same purpose, is in fact suggested by commutator theory and confirmed by what Amitsur says in Section 1 of \cite{[A1954b]}. We compare it (Theorem 6.13) with \textit{local solvability}, which we also introduce, inspired by a more ring-theoretic story, where again, the word ``solvability'' is never used. The aim here is to gain a better understanding of the radical closure operator, trying, in particular, to avoid what we call the weak Kaplansky condition (since, e.g. it does not hold in the lattice of ideals of a non-commutative ring; see Subsection 12.7 again), and what we get is:
\begin{itemize}
	\item Theorem 6.13 says that solvability and local solvability agree well, but this requires the weak Kaplansky condition.
	\item Theorem 6.17 describes the radical closure operator in terms of local solvability without requiring the weak Kaplansky condition. The idea here goes back to J. Levitzki (see \cite{[L1951]}, where, however, only the ring-theoretic context is considered).
\end{itemize}

Section 7 introduces another `solvability', related to those previously introduced (Theorem 7.4), which allows us to reformulate condition 4.4(c), and obtain our main result (Theorem 7.7), which can be seen as an improved version of Theorem 4.4. Unfortunately it does use the weak Kaplansky condition, but the reader is supposed to agree, especially with the examples given in Sections 10 and 11, that it is the only condition there that is possibly hard to check. The new solvability of Section 7 is inspired by Amitsur's upper radical construction \cite{[A1952]}, which generalizes the ring-theoretic one.

Section 8 only recalls the definition of an internal pseudogroupoid introduced in \cite{[JP2001]}.

Section 9 begins with a categorical context involving an abstract notion of commutator, far more general than the contexts considered in, say, \cite{[A1954a]}, \cite{[BB1985]}, and \cite{[B1986]}, and explains how complete multiplicative lattices of internal equivalence relations with sober spectra occur there (Theorem 9.2). After that it considers special commutators defined via internal pseudogroupoids as in \cite{[JP2001]}; this context is still far more general than those where prime ideals in categories were considered by other authors. 

Section 10 describes the passage from categories to varieties of universal algebras, and shows that Theorem 7.7 is widely applicable to congruence lattices of congruence modular universal algebras. 

Section 11 presents another application to a wide class of special cases, as its title shows, to what we think should be called ideal lattices in the \textit{commutative world}. This includes results mentioned in Facts 1.3 and 1.4 as very special cases.

Our story suggested many natural further questions, and, in order to avoid making the paper too long, we answered only a few of them, in the form of additional remarks collected in Section 12.

Ending this Introduction, we should apologize to those authors whose work related to our story is either not mentioned properly, or not mentioned at all: since we use so many references with the oldest one from 1924, it was just too hard to draw the line between `closely related' and `less related'. In particular, we don't compare our presentation with the ones of O. Steinfeld \cite{[S1968]} (and \cite{[S1971]}) and of K. Keimel \cite{[K1972]} (except a remark in 12.4(d)), we don't compare our constructions with those of M. Ern\'e \cite{[E2000]}, and we don't discuss any links with the general theory of radicals (see e.g. B. J. Gardner and R. Wiegandt \cite{[GW2004]} and references therein).

\textbf{Acknowledgment.} We are very grateful to Aldo Ursini and Paolo Aglian\`o for suggesting to look at the papers \cite{[U1984]}, \cite{[A1989]}, and \cite{[A1993]}.               

\section{Abstract prime spectra}

\begin{defi}\emph{
    A \textit{complete multiplicative lattice} is a complete lattice $L$ equipped with a multiplication satisfying $xy\leqslant x\wedge y$ for all $x,y\in L$.} 
\end{defi}

\textit{Throughout this paper $L$ will denote a complete multiplicative lattice, whose smallest and largest elements will be denoted by $0$ and $1$, respectively.}

\begin{defi}\emph{
	An element $p\neq1$ in $L$ is said to be \textit{prime} if it satisfies the implication
	\begin{equation*}
	xy\leqslant p\Rightarrow(x\leqslant p\,\,\,\text{or}\,\,\,y\leqslant p);
	\end{equation*}
	the set $\mathrm{Spec}(L)$ of all such elements in $L$ will be called the \textit{Zariski spectrum} of $L$.}
\end{defi}

\begin{remark}\emph{
	As follows from our definitions, an element $p\neq1$ in $L$ is prime if and only if
		\begin{equation*}
		xy\leqslant p\Leftrightarrow(x\leqslant p\,\,\,\text{or}\,\,\,y\leqslant p).
		\end{equation*}}
\end{remark}

For any $x\in L$ we put
\begin{equation*}
\mathrm{V}(x)=\{p\in\mathrm{Spec}(L)\,|\,x\leqslant p\},
\end{equation*}    
and, for all $x,y\in L$ and $S\subseteq L$, we have: 
\begin{itemize}
	\item $\mathrm{V}(1)=\emptyset$.
	\item $\mathrm{V}(xy)=\mathrm{V}(x)\cup\mathrm{V}(y)$. Indeed, we have
	\begin{equation*}
	p\in\mathrm{V}(xy)\Leftrightarrow xy\leqslant p\Leftrightarrow(x\leqslant p\,\,\,\text{or}\,\,\,y\leqslant p)\Leftrightarrow p\in\mathrm{V}(x)\cup\mathrm{V}(y).
	\end{equation*}
	\item $\mathrm{V}(\bigvee S)=\bigcap_{s\in S}\mathrm{V}(s)$.
\end{itemize}

This allows us to define a topology on $\mathrm{Spec}(L)$ by choosing closed sets to be all sets of the form $\mathrm{V}(x)$ $(x\in L)$.  This topology is to be called \textit{Zariski topology}, and from now on we assume that $\mathrm{Spec}(L)$ is equipped with this topology, and call this space the \textit{prime spectrum} of $L$.
\begin{defi}\emph{
	An element in $L$ is said to be a \textit{radical element} if it can be presented as a meet of prime elements. For an arbitrary element $x$ in $L$, we will write $\sqrt{x}$ for the smallest radical element $r$ in $L$ with $x\leqslant r$, and call this element the radical of $x$.}
\end{defi}
It is obvious that $\sqrt{-}:L\to L$ is a closure operator, that is:
\begin{itemize}
	\item $x\leqslant y\Rightarrow\sqrt{x}\leqslant\sqrt{y}$,
	\item $x\leqslant\sqrt{x}$,
	\item $\sqrt{\sqrt{x}}=\sqrt{x}$,
\end{itemize}
for all $x,y\in L$. Furthermore, $\sqrt{x}=x$ if and only if $x$ is a radical element, and from Definition 2.4, we immediately obtain:
\begin{prop}
	$\mathrm{V}(\sqrt{x})=\mathrm{V}(x)$ and $\sqrt{x}\leqslant\sqrt{y}\Leftrightarrow\mathrm{V}(y)\subseteq\mathrm{V}(x)$ for all $x,y\in L$. Furthermore, if $p$ is prime, then $\mathrm{V}(p)$ is the closure of $\{p\}$.\qed
\end{prop}
\begin{lem}
	$\mathrm{Spec}(L)$ is a sober space. That is:
	\begin{itemize}
		\item [(a)] $\mathrm{Spec}(L)$ is a $\mathrm{T}_0$-space;
		\item [(b)] every irreducible closed subset of $\mathrm{Spec}(L)$ is the closure of a one-element set.
	\end{itemize}	
\end{lem}
\begin{proof}
	(a): Just note that if $p\nleqslant q$ in $\mathrm{Spec}(L)$, we have $q\in\mathrm{V}(q)\setminus\mathrm{V}(p)$.
	
	(b): Let $M$ be an irreducible closed subset of $\mathrm{Spec}(L)$. Without loss of generality we can assume that $M=\mathrm{V}(p)$, where $p\neq1$ and $p$ is a radical element of $L$. Suppose $xy\leqslant p$. Then $\mathrm{V}(p)\subseteq\mathrm{V}(xy)=\mathrm{V}(x)\cup\mathrm{V}(y)$, and, since $\mathrm{V}(p)$ is irreducible, we have $\mathrm{V}(p)\subseteq\mathrm{V}(x)$ or $\mathrm{V}(p)\subseteq\mathrm{V}(y)$. If the first of these inclusions holds, then, for every prime $q$, we have $p\leqslant q\Rightarrow x\leqslant q$; and, since $p$ is a radical element, this gives $x\leqslant p$. Similarly, the second inclusion implies $y\leqslant p$. That is, $p$ is prime. Since $p$ is prime, $\mathrm{V}(p)$ is the closure of $\{p\}$.
\end{proof}

\section{Some relevant categories and functors}

Let $X$ and $Y$ be complete lattices. Considering $X$ and $Y$ as categories we can speak of an adjunction $(f,u):X\to Y$, that is, order preserving maps
\begin{equation*}
\xymatrix{X\ar@<0.5ex>[r]^-f&Y\ar@<0.5ex>[l]^-u\,\,\,\text{with}\,\,\,f(x)\leqslant y\Leftrightarrow x\leqslant u(y)}
\end{equation*}
for all $x\in X$ and $y\in Y$. Such adjunctions between lattices are known under various names, ``covariant Galois connections'', and others. It is also well known (independently of the multiplicative structure) that:
\begin{itemize}
	\item $f$ and $u$ completely determine each other;
	\item $f$ preserves arbitrary joins, and, conversely, any join preserving map\\ $f:X\to Y$ is the first component of such an adjunction;
	\item $u$ preserves arbitrary meets, and, conversely, any meet preserving map\\ $u:Y\to X$ is the second component of such an adjunction.
\end{itemize}
\begin{defi}\emph{
	 An adjunction $(f,u):X\to Y$ between complete multiplicative lattices is said to be \textit{compatible} if $f(1)=1$ and $f(x)f(x')\leqslant f(xx')$ for all $x,x'\in X$. The category $\mathsf{CML}$ of complete multiplicative lattices is defined as the category whose morphisms are compatible adjunctions. Furthemore, we will say that the adjunction above is \textit{strictly compatible} if $f(x)f(x')=f(xx')$ for all $x,x'\in X$, and the corresponding wide subcategory of $\mathsf{CML}$ will be denoted by $\mathsf{CML}_\mathrm{s}$.}    
\end{defi}
\begin{teo}
	Let $(f,u):X\to Y$ be a morphism in $\mathsf{CML}$. Then:
	\begin{itemize}
		\item [(a)] if $p$ is a prime in $Y$, then $u(p)$ is prime in $X$;
		\item [(b)] the map $\mathrm{Spec}(f,u):\mathrm{Spec}(Y)\to\mathrm{Spec}(X)$ defined by $\mathrm{Spec}(f,u)(p)=u(p)$ is continuous, moreover, $\mathrm{Spec}(f,u)^{-1}(\mathrm{V}(x))=\mathrm{V}(f(x))$ for each $x\in X$;
		\item [(c)] the assignment above determines a functor $\mathrm{Spec}:\mathsf{CML}^{\mathrm{op}}\to\mathsf{STop}$, where $\mathsf{STop}$ is the category of sober topological spaces.
	\end{itemize}
\end{teo}
\begin{proof}
	(a): For $x,x'\in X$, we have:
	\begin{equation*}
	xx'\leqslant u(p)\Leftrightarrow f(xx')\leqslant p\Rightarrow f(x)f(x')\leqslant p
	\end{equation*}
	\begin{equation*}
	\Rightarrow(f(x)\leqslant p\,\,\text{or}\,\,f(x')\leqslant p)\Leftrightarrow(x\leqslant u(p)\,\,\text{or}\,\,x'\leqslant u(p)),
	\end{equation*}
	and $u(p)\neq1$ since 
	\begin{equation*}
	u(p)=1\Leftrightarrow1\leqslant u(p)\Leftrightarrow f(1)\leqslant p\Leftrightarrow 1\leqslant p\Leftrightarrow p=1.
	\end{equation*}
	
	(b): For $x\in X$, we have: 
    \begin{equation*}
    \mathrm{Spec}(f,u)^{-1}(\mathrm{V}(x))=\{p\in\mathrm{Spec}(Y)\,|\,u(p)\in\mathrm{V}(x)\}
    \end{equation*}
    \begin{equation*}
    =\{p\in\mathrm{Spec}(Y)\,|\,x\leqslant u(p)\}=\{p\in\mathrm{Spec}(Y)\,|\,f(x)\leqslant p\}=\mathrm{V}(f(x)).
    \end{equation*}

    (c) follows from Lemma 2.6 and (b).
\end{proof}
\begin{example}\emph{
	$L$ is a quantale (see \cite{[R1990]}, which refers to \cite{[M1986]}) if and only if its multiplication is:
	\begin{itemize}
		\item [(a)] associative, that is, $x(yz)=(xy)z$ for all $x,y,z\in L$;
		\item [(b)] infinitary-distributive, that is $x(\bigvee S)=\bigvee_{s\in S}xs$ and $(\bigvee S)x=\bigvee_{s\in S}sx$ for all $x\in L$ and $S\subseteq L$.
	\end{itemize}
	Note, however, that the inequality $xy\leqslant x\wedge y$ we required is not required for quantales in general; it holds if and only if all elements of the quantale are two-sided (in the terminology of quantale theory, where it means that $1x=x=x1$ for all $x\in L$). When $L$ and $M$ are quantales, to say that $(f,u):L\to M$ is a compatible adjunction is the same as to say that $f:L\to M$ is a closed sup-map of quantales with $f(1)=1$; and then to say that $(f,u):L\to M$ is strictly compatible is the same as to say that $f:L\to M$ is a homomorphism of quantales (see \cite{[R1990]}) with $f(1)=1$. In particular, when $L$ and $M$ are frames (=locales), that is, when their multiplication coincides with the meet operation, all the compatible adjunctions between them become strictly compatible, and they become the same as frame homomorphisms.}   
\end{example}
\begin{lem}
	$\sqrt{x}\wedge\sqrt{y}=\sqrt{xy}$ for all $x,y\in L$.
\end{lem}
\begin{proof}
	$xy\leqslant x\wedge y$ gives $xy\leqslant x$ and $xy\leqslant y$ and then $\sqrt{xy}\leqslant \sqrt{x}$ and $\sqrt{xy}\leqslant \sqrt{y}$. Therefore $\sqrt{xy}\leqslant \sqrt{x}\wedge\sqrt{y}$. To prove the opposite inequality it suffices to prove that $\sqrt{x}\wedge\sqrt{y}\leqslant p$ for every prime $p$ with $xy\leqslant p$. But the last inequality gives $x\leqslant p$ or $y\leqslant p$, and then $\sqrt{x}\leqslant p$ or $\sqrt{y}\leqslant p$.
\end{proof}
\begin{remark}\emph{
	Let $\sqrt{L}$ be the ordered set of all radical elements of $L$. Then:
	\begin{itemize}
		\item [(a)] As follows from Proposition 2.5, the assignment $x\mapsto-\mathrm{V}(x)$ determines an isomorphism between $\sqrt{L}$ and the complete lattice $\Omega(\mathrm{Spec}(L))$ of open subsets of $\mathrm{Spec}(L)$. In particular, $\sqrt{L}$ is a frame, and we will consider it as a complete multiplicative lattice whose multiplication is the meet operation.
		\item [(b)] For $x\in L$ and $y\in \sqrt{L}$, we have $\sqrt{x}\leqslant y\Leftrightarrow x\leqslant y$, which determines the adjunction $(\rho,\iota):L\to\sqrt{L}$, defined by $\rho(x)=\sqrt{x}$ and (accordingly) $\iota(y)=y$. This adjunction is strictly compatible, which follows from Lemma 3.4.  
		\item [(c)] Let us compare the sets $\mathrm{Spec}(L)$ and $\mathrm{Spec}(\sqrt{L})$. If $p$ is a prime element in $L$, then $p=\sqrt{p}$, and it is a prime element in $\mathrm{Spec}(\sqrt{L})$. Indeed, if $x,y\in\mathrm{Spec}(\sqrt{L})$ have $x\wedge y\leqslant p$ in $\mathrm{Spec}(\sqrt{L})$, then $xy\leqslant x\wedge y\leqslant p$ in $L$, and so $x\leqslant p$ or $x\leqslant p$. Conversely, if $p$ is a prime element in $\sqrt{L}$, then it is a prime element in $L$, as follows from Theorem 3.2(a) applied to $(\rho,\iota):L\to\sqrt{L}$. That is, $\mathrm{Spec}(L)$ and $\mathrm{Spec}(\sqrt{L})$ are the same sets. Moreover, the first equality of Proposition 2.5 tells us that they are the same topological spaces. Furthermore, the identity map $\mathrm{Spec}(\sqrt{L})\to\mathrm{Spec}(L)$ is exactly the image of the morphism $(\rho,\iota):L\to\sqrt{L}$ under the functor $\mathrm{Spec}:\mathsf{CML}^{\mathrm{op}}\to\mathsf{STop}$, as follows from the fact that $\iota:\sqrt{L}\to L$ is the inclusion map. 
		\item [(d)] However, $\sqrt{L}$ is not necessarily closed under the multiplication in $L$. For example, let $L$ be a three-element monoid of the form $\{1,x,x^2\}$ with $x^3=x^2$ and ordered by $x^2<x<1$. Then $\sqrt{L}=\{1,x\}$ and
		\begin{equation*}
		\mathrm{Spec}(L)=\{x\}=\mathrm{Spec}(\sqrt{L}),
		\end{equation*}
		but $x^2$ is missing in $\sqrt{L}$.  
	\end{itemize}}
\end{remark}

The standard way to associate a sober topological space to a frame is a restriction of our functor $\mathrm{Spec}:\mathsf{CML}^{\mathrm{op}}\to\mathsf{STop}$, and, in particular, our Lemma 2.6 is an immediate consequence of Lemma 1.7 of Chapter II in \cite{[J1982]}, the isomorphism $\sqrt{L}\approx\Omega(\mathrm{Spec}(L))$, and the equality $\mathrm{Spec}(L)=\mathrm{Spec}(\sqrt{L})$. Furthermore, what we learn from Chapter II in \cite{[J1982]} includes the diagram
\begin{equation*}
\xymatrix{\mathsf{Loc}\ar[dr]\ar@<0.5ex>[r]^{\mathrm{pt}}&\mathsf{Top}\ar[dl]\ar@<0.5ex>[l]^\Omega\\\mathsf{SLoc}\ar[u]\ar@<0.5ex>[r]&\mathsf{STop}\ar[u]\ar@<0.5ex>[l]\\\mathsf{CLoc}\ar[u]\ar@<0.5ex>[r]&\mathsf{CTop}\ar[u]\ar@<0.5ex>[l]}
\end{equation*}
in which:
\begin{itemize}
	\item $\mathsf{Loc}$ is the category of locales, which is the same as the opposite category of frames, and which is a full subcategory of $\mathsf{CML}^{\mathrm{op}}$ and of $\mathsf{CML}_{\mathrm{s}}^{\mathrm{op}}$.
	\item $\mathsf{Top}$ is the category of topological spaces.
	\item All the vertical arrows are full subcategory inclusion functors.
	\item The functor $\mathsf{Loc}\to\mathsf{STop}$ is essentially surjective on objects and it can be identified with the restriction of the functor $\mathrm{Spec}:\mathsf{CML}^{\mathrm{op}}\to\mathsf{STop}$ on $\mathsf{Loc}$; and $\mathrm{pt}$, the `functor of points', is its composite with the inclusion functor $\mathsf{STop}\to\mathsf{Top}$. 
	\item $\Omega:\mathsf{Top}\to\mathsf{Loc}$ carries topological spaces to their locales of open sets. The locales that occur this way (up to isomorphism) are called \textit{spatial}, and their category is denoted by $\mathsf{SLoc}$. Therefore $\Omega$ is the composite of the essentially-surjective-on-objects functor $\mathsf{Top}\to\mathsf{SLoc}$ it induces with the inclusion functor $\mathsf{SLoc}\to\mathsf{Loc}$.
	\item $(\Omega,\mathrm{pt})$ forms an adjunction, and the second row of our diagram is the largest category equivalence it induces. That is, the classes of objects of $\mathsf{SLoc}$ and $\mathsf{STop}$ can also be defined using suitable canonical morphism as
	\begin{center} $\{L\in\mathsf{Loc}\,|\,\Omega(\mathrm{pt}(L))\to L$ is a local isomorphism$\}$, and\\ $\{X\in\mathsf{Top}\,|X\to\mathrm{pt}(\Omega(L))$ is a homeomorphism$\}$, \end{center}
	respectively.
	\item Let us recall several definitions:
	
	An element $x$ of a complete lattice is said to be \textit{compact}\footnote{According to \cite{[B1973]}, this notion was first introduced in \cite{[N1949]}, while the closely related notion of \textit{join-inaccessible} (which we will use in the next section) was introduced in \cite{[BF1948]}; \cite{[J1982]} and some other papers say ``finite'' instead of ``compact''.} if every subset $S$ of that lattice with $x\leqslant\bigvee S$ has a finite subset $F$ with $x\leqslant\bigvee F$. The lattice itself is said to be \textit{compact} if its largest element 1 is compact. A complete lattice is said to be \textit{algebraic} if every element in it is a join of compact elements. A spatial locale is said to be \textit{coherent} if it is algebraic and the set of its compact elements forms a sublattice (with $0$ and $1$), or, equivalently, a $\wedge$-subsemilattice (with $1$) in it. A sober topological space is said to be \textit{coherent} if so is the frame $\Omega(X)$; for topological spaces, we will usually say ``spectral'' instead of ``coherent''. Note that in purely topological terms a spectral space is a compact sober topological space in which compact open subsets form a basis of topology that is closed under finite intersections.
	
	$\mathsf{CLoc}$ and $\mathsf{CTop}$ denote the categories of coherent locales and coherent (=spectral) topological spaces, respectively. The third row of our diagram is the category equivalence induced by the equivalence displayed as the second row. Note, however, that there are good reasons to restrict morphisms in $\mathsf{CLoc}$ and in $\mathsf{CTop}$ to so-called coherent ones (see e.g. \cite{[J1982]}).   
\end{itemize}

The following theorem is a consequence of the equivalence between $\mathsf{CLoc}$ and $\mathsf{CTop}$, the isomorphism $\sqrt{L}\approx\Omega(\mathrm{Spec}(L))$, and the equality $\mathrm{Spec}(L)=\mathrm{Spec}(\sqrt{L})$:
\begin{teo}
	$\mathrm{Spec}(L)$ is a spectral space if and only if the following conditions hold:
	\begin{itemize}
		\item [(a)] $\sqrt{L}$ is compact;
		\item [(b)] $\sqrt{L}$ is algebraic;
		\item [(c)] if $x$ and $y$ are compact elements in $\sqrt{L}$, then so is $x\wedge y$.
	\end{itemize}
\end{teo}
\begin{example}\emph{
	Let $\mathrm{Com}(L)$ be the complete multiplicative lattice obtained from $L$ by taking the same complete lattice and replacing the multiplication of $L$ with the multiplication $*$ defined by $x*y=xy\vee yx$. Since $xy\leqslant x*y$, this gives a compatible adjunction $(f,u):\mathrm{Com}(L)\to L$, in which $f$ and $u$ are the identity maps; informally $f=1_L=u$. Furthermore, the following conditions are obviously equivalent:
	\begin{itemize}
		\item [(a)] the adjunction above is strictly compatible;
		\item [(b)] the adjunction above is an isomorphism in $\mathsf{CML}$ (or, equivalently, in $\mathsf{CML}_\mathrm{s}$);
		\item [(c)] $\mathrm{Com}(L)=L$.
		\item [(d)] the multiplication of $L$ is commutative.
	\end{itemize}   
	Note also that if $L$ satisfies the \textit{monotonicity condition}
	\begin{equation*}
	(x\leqslant y\,\,\&\,\,x'\leqslant y')\Rightarrow xx'\leqslant yy'
	\end{equation*}
	(for all $x,y\in L$), then for $p\in L$, we have:
	\begin{equation*}
	p\,\,\text{is prime in}\,\,L\Leftrightarrow p\,\,\text{is prime in}\,\,\mathrm{Com}(L).
	\end{equation*} 
	The implication ``$\Rightarrow$'' is obvious, while the implication `` $\Leftarrow$'' can be proved as follows:}
	
	\emph{Suppose $p$ is prime in $\mathrm{Com}(L)$ and $xy\leqslant p$. Then we have:
	\begin{equation*}
	yx*yx=(yx)(yx)\vee(yx)(yx)\leqslant xy\vee xy=xy\leqslant p,
	\end{equation*}
	and so $yx\leqslant p$, after which we can write $x*y=xy\vee yx\leqslant p$ and conclude that $x\leqslant p$ or $y\leqslant p$.}
	
	\emph{Furthermore, the equality $\mathrm{Spec}(L)=\mathrm{Spec}(\mathrm{Com}(L))$ tells us that $\mathrm{Spec}(f,u)$ (where $(f,u):\mathrm{Com}(L)\to L$) is as above) is the identity homeomorphism.}
\end{example}

\section{Involving algebraic radicals}

\begin{defi}\emph{
	We will say that $L$ \textit{has algebraic radicals} if $\sqrt{-}:L\to L$ is an algebraic closure operator, that is, if it preserves directed joins.}
\end{defi}
\begin{prop}
	Suppose $L$ has algebraic radicals. If $x$ is a compact element of $L$, then $\sqrt{x}$ is a compact element of $\sqrt{L}$. In particular, if $L$ is compact, then so is $\sqrt{L}$.
\end{prop}
\begin{proof}
	Since $\sqrt{-}$ is a closure operator, the join in $\sqrt{L}$ of a subset $S$ of $\sqrt{L}$ is the same as $\sqrt{\bigvee S}$, where $\bigvee S$ is the join of $S$ in $L$. Therefore what we need to prove (for a compact $x$ in $L$) is that for every $S\subseteq\sqrt{L}$ with $\sqrt{x}\leqslant\sqrt{\bigvee S}$ there exists a finite subset $F$ of $S$ with $\sqrt{x}\leqslant\sqrt{\bigvee F}$. For, we take $T$ to be the set of all finite joins of elements of $S$, which makes $T$ a directed set whose join is the same as the join of $S$. This gives
	\begin{equation*}
	\sqrt{x}\leqslant\sqrt{\bigvee S}=\sqrt{\bigvee T}=\bigvee_{t\in T}\sqrt{t}.
	\end{equation*}
	Since $x\leqslant\sqrt{x}$, $x$ is compact, and $T$ is closed under finite joins, we easily conclude that $x\leqslant\sqrt{t}$ and then that $\sqrt{x}\leqslant\sqrt{t}$ for some $t\in T$. This gives the desired inequality for $F$ being any finite set of elements of $S$ whose join is $t$.  
\end{proof}
From this proposition and the isomorphism $\sqrt{L}\approx\Omega(\mathrm{Spec}(L))$, we obtain:
\begin{cor}
	Suppose $L$ has algebraic radicals. If $x$ is a compact element of $L$, then $-\mathrm{V}(x)$ is a compact open subset of $\mathrm{Spec}(L)$. In particular, if $L$ is compact, then so is $\mathrm{Spec}(L)$.\qed
\end{cor}
\begin{teo}
	$\mathrm{Spec}(L)$ is a spectral space whenever the following conditions hold:
	\begin{itemize}
		\item [(a)] $L$ is compact;
		\item [(b)] $L$ is algebraic;
		\item [(c)] $L$ has algebraic radicals;
		\item [(d)] if $x$ and $y$ are compact elements in $L$, then there exists a compact $c\in L$ with $\sqrt{c}=\sqrt{xy}$.
	\end{itemize}
\end{teo}    
\begin{proof}
	Indeed:
	\begin{itemize}
		\item $\mathrm{Spec}(L)$ is sober by Lemma 2.6.
		\item The fact that $\mathrm{Spec}(L)$ is compact follows from (a), (c), and Corollary 4.3. 
		\item For every compact $x$, $-\mathrm{V}(x)$ is a compact open subset of $\mathrm{Spec}(L)$ by (c) and Corollary 4.3.
	\end{itemize}
	After that it remains to prove that:
	\begin{itemize}
		\item If $x$ and $y$ are compact, then $-\mathrm{V}(x)\cap-\mathrm{V}(y)$ is a compact open subset of $\mathrm{Spec}(L)$.
		\item Every open subset of $\mathrm{Spec}(L)$ is a union of subsets of the form $-\mathrm{V}(x)$ with compact $x$. 
	\end{itemize}
	The first of these assertions follows from
	\begin{equation*}
	-\mathrm{V}(x)\cap-\mathrm{V}(y)=-(\mathrm{V}(x)\cup\mathrm{V}(y))=-\mathrm{V}(xy)=-\mathrm{V}(\sqrt{xy})=-\mathrm{V}(\sqrt{c})=-\mathrm{V}(c)
	\end{equation*}
	(where $c$ is as in (d)) and Corollary 4.3. The second one follows from the fact that $x=\bigvee S$ implies
	\begin{equation*}
		-\mathrm{V}(x)=-\mathrm{V}(\bigvee S)=-\bigcap_{s\in S}\mathrm{V}(s)=\bigcup_{s\in S}(-\mathrm{V}(s)), 
	\end{equation*}
	(b), and Corollary 4.3 again. 
\end{proof}
\begin{remark}\emph{
	Applying Theorem 4.4 to $\sqrt{L}$ and having in mind that $\mathrm{Spec}(\sqrt{L})=\mathrm{Spec}(L)$ and $\sqrt{x}=x$ for all $x\in\sqrt{L}$, we obtain exactly the ``if'' part of Theorem 3.6. However, unlike Theorem 3.6, Theorem 4.4 gives only sufficient conditions for the space $\mathrm{Spec}(L)$ to be spectral. Indeed, take $L$ to be any complete lattice with $xy=0$ for all $x,y\in L$. Then $\mathrm{Spec}(L)$ being empty is trivially spectral. In this case conditions (a) and (b) of Theorem 4.4 do not hold in general of course.}  
\end{remark}

\section{How to get enough primes?}

The results of this section are very simple and might be called known, at least in special cases; nevertheless we state and prove them not having a convenient reference. ``Known'' especially applies to Definition 5.1, Proposition 5.2, and the text between them, since the multiplication of $L$ plays no role there.
\begin{defi}\emph{
	An element $m\neq1$ in $L$ is said to be
	\begin{itemize}
		\item [(a)] \textit{maximal}, if $m\neq 1$ and $m$ satisfies the implication
		\begin{equation*}
		m<x\Rightarrow x=1;
		\end{equation*}
		\item [(b)] \textit{join-inaccessible}, if, for any directed subset $S$ of $L$, it satisfies the implication
		\begin{equation*}
		m=\bigvee S\Rightarrow\exists_{s\in S}\,m=s.
		\end{equation*}
	\end{itemize}}
\end{defi}

Every compact element in $L$ is obviously join-inaccessible. On the other hand, if $x<y$ and $y$ is join-inaccessible, then there exists $x'\in L$ maximal with the property $x\leqslant x'<y$; this immediately follows from Zorn's Lemma applied to the set of all elements with that property. In particular, applying this to $y=1$, we obtain:

\begin{prop}
	If $L$ is compact, then, for every $x\neq 1$, there exists a maximal element $m$ in $L$ with $x\leqslant m$.
\end{prop}

Next, we introduce:
\begin{defi}\emph{
	We will say that $L$ \textit{has enough primes} if, for every $x\neq1\in L$, there exists a prime $p\in L$ with $x\leqslant p$.} 
\end{defi}
\begin{defi}\emph{
    We will say that $L$ is \textit{distributive} if $x(y\vee z)=xy\vee xz$ and $(x\vee y)z=xz\vee yz$ for all $x$, $y$, and $z$ in $L$.} 
\end{defi}

Note: obviously, this distributivity is not the same as distributivity of $L$ as merely a lattice.

For $x\in L$, although the multiplication of $L$ is not required to be associative, $x^2=xx$ is well defined of course, and, in particular, we have $1^2$ not necessarily equal to $1$.
\begin{prop}
	If $L$ has enough primes, then $1^2=1$ in $L$.
\end{prop}
\begin{proof}
	For a prime $p$, $1^2\leqslant p$ would give $1\leqslant p$, which is a contradiction.
\end{proof}
\begin{prop}
	Suppose $L$ is compact. Then:
	\begin{itemize}
		\item [(a)] $L$ has enough primes if and only if all its maximal elements are prime;
		\item [(b)] if $L$ is distributive and $1^2=1$ in it, then every maximal element in $L$ is prime.
	\end{itemize}
\end{prop}
\begin{proof}
	(a): The ``if'' part follows from Proposition 5.2, while the ``only if'' is obvious.
	
	(b): Suppose $m$ is not prime, and so $xy\leqslant m$, $x\nleqslant m$, and $y\nleqslant m$ for some $x$ and $y$ in $L$. If $m$ is maximal, then $m\vee x=1=m\vee y$, and so
	\begin{equation*}
	1=1^2=(m\vee x)(m\vee y)=mm\vee xm\vee my\vee xy\leqslant m\vee m\vee m\vee m=m,
	\end{equation*}
	which is a contradiction.
\end{proof}
From Propositions 5.5 and 5.6, we immediately obtain:
\begin{teo}
	If $L$ is compact and distributive, then the following conditions are equivalent:
	\begin{itemize}
		\item [(a)] $L$ has enough primes;
		\item [(b)] $1^2=1$ in $L$.
	\end{itemize}
\end{teo}
\begin{remark}\emph{
	Consider the following conditions of $L$ (cf. Example 2.9):
	\begin{itemize}
		\item [(i)] $xy=0$ for all $x,y\in L$;
		\item [(ii)] $1^2=0$ in $L$;
		\item [(iii)] $1^2=1$ in $L$;
		\item [(iv)] $1x=x=x1$ for all $x\in L$;
		\item [(v)] $xy=x\wedge y$ for all $x,y\in L$.
	\end{itemize}
	Note that:
	\begin{itemize}
		\item [(a)] We obviously have $(\mathrm{i})\Rightarrow(\mathrm{ii})$ and $(\mathrm{v})\Rightarrow(\mathrm{iv})\Rightarrow(\mathrm{iii})$. Also $(\mathrm{ii})\Rightarrow(\mathrm{i})$ under the monotonicity condition.
		\item [(b)] We already know what happens under condition $(\mathrm{i})$ (see the example at the end of Remark 4.5), and the same would happen under condition $(\mathrm{ii})$.
		\item [(c)] Suppose $L$ is compact and distributive. Then, according to Theorem 5.7, condition $(\mathrm{iii})$ is exactly what is needed to make all $\mathrm{V}(x)$ non-empty, except $\mathrm{V}(1)$, which must be empty in any case. 
		\item [(d)] The stronger conditions $(\mathrm{iv})$ and $(\mathrm{v})$ will appear (again) later and we will see examples where $(\mathrm{iii})$ holds but $(\mathrm{iv})$ does not, and other examples where $(\mathrm{iv})$ holds but $(\mathrm{v})$ does not. We have already mentioned, in the case of quantales, that $(\mathrm{iv})$ has a name (``two-sided'') and that $(\mathrm{v})$ is what makes a quantale a locale (=frame) (see Example 3.3).
	\end{itemize}}
\end{remark}

\section{Solvable and locally solvable elements}

\begin{defi}\emph{
	For $x\in L$, the \textit{derived series} $x^{(0)},x^{(1)},x^{(2)},\ldots$ is defined by induction as
	\begin{equation*}
	x^{(0)}=x,\,\,x^{(n+1)}=(x^{(n)})^2\,(=(x^2)^{(n)}),
	\end{equation*}
	and, for $y\in L$, we will say that $x$ is $y$-\textit{solvable} if $x^{(n)}\leqslant y$ for some natural $n$. The join of all  $y$-solvable elements of $L$ will be denoted by $\mathrm{solv}(y)$.}
\end{defi}
\begin{prop}
	For elements $x$ and $y$ in $L$, if $x$ is $y$-solvable, then $x\leqslant\sqrt{y}$; hence $\mathrm{solv}(y)\leqslant\sqrt{y}$ for all $y\in L$.\qed 
\end{prop}
\begin{defi}\emph{
	We will say that $L$ satisfies the \textit{weak monotonicity condition} if
	\begin{equation*}
	x\leqslant y\Rightarrow x^2\leqslant y^2
	\end{equation*}
	for all $x,y\in L$.}
\end{defi}
\begin{prop}
	Suppose $L$ satisfies the weak monotonicity condition. Then, for $x,y,z\in L$, we have:
	\begin{itemize}
		\item [(a)] if $x\leqslant y$ and $y$ is $z$-solvable, then $x$ is $z$-solvable;
		\item [(b)] if $x$ is $y$-solvable and $y$ is $z$-solvable, then $x$ is $z$-solvable.\qed  
	\end{itemize}
\end{prop}

From Proposition 6.4(a), we obtain:
\begin{cor}
	If $L$ satisfies the weak monotonicity condition and is algebraic, then $\mathrm{solv}(y)$ is the join of all compact $y$-solvable elements of $L$.
\end{cor}
\begin{prop}
	If $L$ is distributive, then $(x\vee y)^{(2n)}\leqslant x^{(n)}\vee y^{(n)}$ for all $x,y\in L$.
\end{prop}
\begin{proof}
	This inequality is trivial for $n=0$; assuming that it holds for $n$, we have:
	\begin{equation*}
	(x\vee y)^{(2(n+1))}=(x\vee y)^{(2n+1+1)}=(((x\vee y)^{(2n)})^2)^2\leqslant((x^{(n)}\vee y^{(n)})^2)^2
	\end{equation*}
	\begin{equation*}
	=((x^{(n)})^2\vee(x^{(n)})(y^{(n)})\vee(y^{(n)})(x^{(n)})\vee(y^{(n)})^2)^2
	\end{equation*}
	\begin{equation*}
	=(x^{(n+1)}\vee(x^{(n)})(y^{(n)})\vee(y^{(n)})(x^{(n)})\vee y^{(n+1)})^2.
	\end{equation*}
	Then, opening parentheses we obtain a join of sixteen elements, and we have to show that each of them is less or equal to $x^{(n+1)}\vee y^{(n+1)}$. This is certainly true for all those elements that are multiples of either $x^{(n+1)}$ or $y^{(n+1)}$, and so we only need to check the remaining four members, which are:
	\begin{equation*}
	((x^{(n)})(y^{(n)}))^2,\,((x^{(n)})(y^{(n)}))((y^{(n)})(x^{(n)})),\,((y^{(n)})(x^{(n)}))((x^{(n)})(y^{(n)})),
	\end{equation*}
	and $((y^{(n)})(x^{(n)}))^2$. However, each of them is less or equal to, say, $((x^{(n)}))^2=x^{(n+1)}$.
\end{proof}
\begin{cor}
	If $L$ is distributive, and $x$ and $y$ are $z$-solvable in $L$, then $x\vee y$ is solvable.\qed
\end{cor}
\begin{cor}
	If $L$ is distributive, $x\in L$, and $\mathrm{solv}(x)$ is compact in $L$, or, more generally, in the lattice $\{y\in L\,|\,x\leqslant y\}$, then $\mathrm{solv}(x)$ is the largest $x$-solvable element of $L$.\qed
\end{cor}
\begin{defi}\emph{
	For $x,y\in L$, we will say that $x$ is \textit{locally $y$-solvable} if, for every infinite sequence $c_0,c_1,c_2,\ldots$ of compact elements in $L$ with\footnote{Note that every finite such sequence can be made infinite by adding zero's.}
	\begin{equation*}
	c_0\leqslant x\,\,\,\text{and}\,\,\,c_{n+1}\leqslant c_n^2\,(n=0,1,2,\ldots),
	\end{equation*}
	there exists $n$ with $c_n\leqslant y$. The join of all locally $y$-solvable elements of $L$ will be denoted by $\mathrm{loc.solv}(y)$.}
\end{defi}
\begin{lem}
	If $x$ is locally $y$-solvable and $x'\leqslant x$, then $x'$ is locally $y$-solvable.\qed
\end{lem}
\begin{cor}
	If $L$ is algebraic, then $\mathrm{loc.solv}(y)$ is the join of all compact locally $y$-solvable elements of $L$.\qed
\end{cor}
\begin{defi}\emph{
	We will say that $L$ satisfies the \textit{weak} (form of) \textit{Kaplansky condition} if, whenever $x\in L$ is compact, so is $x^2$.}
\end{defi}
\begin{teo}
	For $x,y\in L$, we have:
	\begin{itemize}
		\item [(a)] if $L$ satisfies the weak monotonicity condition and $x$ is $y$-solvable, then $x$ is locally $y$-solvable;
		\item [(b)] if $L$ satisfies the weak Kaplansky condition and $x$ is compact and locally $y$-solvable, then $x$ is $y$-solvable.
		\item [(c)] if $L$ is algebraic and satisfies the weak monotonicity condition and the weak Kaplansky condition, then $\mathrm{loc.solv}(y)=\mathrm{solv}(y)$ for all $y\in L$.
	\end{itemize}
\end{teo}
\begin{proof}
	(a): For a sequence $c_0,c_1,c_2,\ldots$ of compact elements in $L$ with $c_0\leqslant x$ and $c_{n+1}\leqslant c_n^2\,(n=0,1,2,\ldots)$, we can prove by induction that $c_n\leqslant x_{(n)}$ for all $n$. Indeed, this is trivial for $n=0$, and, once $c_n\leqslant x_{(n)}$, we have:
	\begin{equation*}
	c_{n+1}\leqslant c_n^2\leqslant x_{(n)}^2=x_{(n+1)}.
	\end{equation*}
	Therefore $x_{(n)}\leqslant y$ implies $c_n\leqslant y$, as desired.
	
	(b): If $x$ is compact, then, applying Definition 6.9 to the sequence $x=x_{(0)},x_{(1)},\\x_{(2)},\ldots$, we conclude that there exists $n$ with $x_{(n)}\leqslant y$.
	
	(c): The inequality $\mathrm{loc.solv}(y)\leqslant\mathrm{solv}(y)$ follows from Corollary 6.11 and (b), while the opposite inequality follows from (a).   
\end{proof}
\begin{prop}
	Suppose $L$ is algebraic. If $x$ is compact and locally $y$-solvable in $L$, then $x\leqslant\sqrt{y}$; hence $\mathrm{loc.solv}(y)\leqslant\sqrt{y}$ for all $y\in L$.
\end{prop}
\begin{proof}
	Suppose $x$ is compact and locally $y$-solvable, and has $x\nleqslant\sqrt{y}$. Then, excluding the trivial case $y=1$, we have $x\nleqslant p$ for some prime $p$ with $y\leqslant p$. Using $p$, let us choose an infinite sequence $c_0,c_1,c_2,\ldots$ of compact elements, with $c_{n+1}\leqslant c_n^2$ and $c_n\nleqslant p$ for all $n$, inductively as follows:
	\begin{itemize}
		\item $c_0=x$.
		\item Once $c_n$ is chosen, $c_n\nleqslant p$ gives $c_n^2\nleqslant p$, since $p$ is prime. Since $L$ algebraic, this tells us that we can indeed choose a compact $c_{n+1}$ with $c_{n+1}\leqslant c_n^2$ and $c_{n+1}\nleqslant p$.
	\end{itemize}
	Then we obtain $c_n\leqslant y$ by Definition 6.9, which is a contradiction since we also have $y\leqslant p$ and $c_n\nleqslant p$.  
\end{proof}
\begin{remark}\emph{
	Requiring $L$ to satisfy the weak Kaplansky condition instead of being algebraic, we would obtain the first inequality of Proposition 6.14 as an immediate consequence of Proposition 6.2 (which itself is obvious) and Theorem 6.13(b). However, to automatically deduce the second inequality of Proposition 6.14 from that, we would still need algebraicity.} 
\end{remark}
\begin{lem}
	Suppose $L$ satisfies the monotonicity condition (introduced in the discussion of Example 3.7) and is algebraic. Then an element $p\neq1$ in $L$ is prime if and only if $p$ is locally prime, by which we mean that
	\begin{equation*}
	ab\leqslant p\Rightarrow(a\leqslant p\,\,\,\text{or}\,\,\,b\leqslant p)
	\end{equation*}
	for all compact $a$ and $b$ in $L$. 
\end{lem}
\begin{proof}
	The ``only if'' part is trivial. To prove the ``if'' part, suppose $p\in L$ is not prime and so $xy\leqslant p$, $x\nleqslant p$, and $y\nleqslant p$ for some $x,y\in L$. Then, since $L$ is algebraic, there exist compact $a,b\in L$ with $a\leqslant x$, $a\nleqslant p$, $b\leqslant y$, and $b\nleqslant p$, which also gives $ab\leqslant xy\leqslant p$ by the monotonicity condition. This tells us that $p$ is not locally prime.
\end{proof}
\begin{teo}
	Suppose $L$ is distributive and $x,y\in L$. If $x$ is compact and $x\leqslant\sqrt{y}$, then $x$ is locally $y$-solvable. Hence, if $L$ is distributive and algebraic, then $\sqrt{y}=\mathrm{loc.solv}(y)$ for all $y\in L$.
\end{teo}
\begin{proof}
	Under the assumptions above, suppose $x$ is not locally $y$-solvable. Then there is a sequence $c_0,c_1,c_2,\ldots$ of compact elements in $L$ with $c_0\leqslant x$, and $c_{n+1}\leqslant c_n^2$ and $c_n\nleqslant y$ for all $n$. Consider the set
	\begin{equation*}
	Y=\{z\in L\,|\,y\leqslant z\,\,\&\,\,\forall_n\,c_n\nleqslant z\}.
	\end{equation*}
	We claim that $Y$ satisfies the assumptions of Zorn's Lemma. Moreover, it is (non-empty and) closed under joins of non-empty chains. Indeed:
	\begin{itemize}
		\item $Y$ is non-empty since it (obviously) contains $y$.
		\item Let $S$ be a non-empty chain in $Y$. If $\bigvee S\not\in Y$, then $c_n\nleqslant\bigvee S$ for some $n$. Since $c_n$ is compact and $S$ is a chain, this gives $c_n\nleqslant s$ for some $s\in S$. Since $S\subseteq Y$, this is a contradiction. Therefore $\bigvee S\in Y$. 
	\end{itemize}
	By Zorn's Lemma $Y$ has a maximal element (obviously not equal to 1), and we claim that any such maximal element $p$ is prime. Indeed, if $p$ is not prime, then:
	\begin{itemize}
		\item By Lemma 6.16, there exist compact $a,b\in L$ with $a\nleqslant p$, $b\nleqslant p$, and $ab\leqslant p$.
		\item Since $a\nleqslant p$ and $b\nleqslant p$, $p$ is strictly smaller than $p\vee a$ and than $p\vee b$.
		\item Since $p$ is a maximal element of $Y$, it follows that $p\vee a$ and than $p\vee b$ are not in $Y$.
		\item Since $y\leqslant p\leqslant p\vee a$ and $y\leqslant p\leqslant p\vee b$, while $p\vee a$ and $p\vee b$ are not in $Y$, we have $c_m\leqslant p\vee a$ and $c_n\leqslant p\vee b$ for some $m$ and $n$. Moreover, since the sequence $c_0,c_1,c_2,\ldots$ is decreasing, we can assume $m=n$.
		\item This gives 
		\begin{equation*}
		c_{n+1}\leqslant c_nc_n\leqslant(p\vee a)(p\vee b)=pp\vee ap\vee pb\vee ab\leqslant p\vee p\vee p\vee p=p,
		\end{equation*}
		which contradicts to $p\in Y$. 
	\end{itemize}
	That is, assuming that $x$ is not locally $y$-solvable, we found that $Y$ has an element $p$, which is prime in $L$. We conclude:
	\begin{itemize}
		\item On the one hand, since $x\leqslant\sqrt{y}$ (by the original assumption), $y\leqslant p$ (since $p\in Y$), and $p$ is prime, we have $x\leqslant p$.
		\item On the other hand, since $c_0\leqslant x$ and $c_0\nleqslant p$ (since $p\in Y$), we have $x\nleqslant p$. 
	\end{itemize}
	This contradiction completes our proof.
\end{proof}

\section{Another approach to solvability, to make radicals algebraic}

Given $x\in L$, we define a transfinite sequence $x_{(0)},x_{(1)},x_{(2)},\ldots\in L$ by induction as
\begin{equation*}
x_{(0)}=x,\,\,\,x_{(\alpha+1)}=\bigvee_{y^2\leqslant x_{(\alpha)}}y,\,\,\,x_{(\lambda)}=\bigvee_{\alpha<\lambda}x_{(\alpha)},\, \text{if } \lambda\,\, \text{is a limit ordinal},
\end{equation*}
and define $\mathrm{Solv}(x)$ as the join of this sequence.
We immediately obtain
\begin{prop}
	$\mathrm{Solv}:L\to L$ is a closure operator with $\mathrm{Solv}(x)\leqslant\sqrt{x}$ for every $x\in L$.
\end{prop}
\begin{proof}
	The properties
	\begin{itemize}
		\item $x\leqslant y\Rightarrow\mathrm{Solv}(x)\leqslant\mathrm{Solv}(x)$,
		\item $x\leqslant\mathrm{Solv}(x)$,
		\item $\mathrm{Solv}(\mathrm{Solv}(x))=\mathrm{Solv}(x)$ 
	\end{itemize}
	obviously hold. To prove the inequality $\mathrm{Solv}(x)\leqslant\sqrt{x}$, it suffices to note that every prime $p\in L$ satisfies the implication $x_{(\alpha)}\leqslant p\Rightarrow x_{(\alpha+1)}\leqslant p$. 
\end{proof}
\begin{lem}
	Suppose $L$ is algebraic and satisfies the weak monotonicity condition. Then, for $x\in L$, $x_{(1)}$ is the join of all compact elements $c\in L$ with $c^2\leqslant x$.
\end{lem}
\begin{proof}
	Since $L$ is algebraic, it suffices to prove that, for every $y\in L$ with $y^2\leqslant x$ and every compact $c\in L$ with $c\leqslant y$, we have $c^2\leqslant x$. But this follows from the weak monotonicity condition.
\end{proof}
\begin{teo}
	Suppose $L$ is as in Theorem 6.13(c), that is, it is algebraic and satisfies the weak monotonicity condition and the weak Kaplansky condition. Then $(x_{(\omega)})_{(1)}=x_{(\omega)}$ (where $\omega$ is the first infinite ordinal), and so $\mathrm{Solv}(x)=x_{(\omega)}$, for every $x\in L$.
\end{teo}
\begin{proof}
	Thanks to Lemma 7.2, it suffices to prove that for every compact $c\in L$ with $c^2\leqslant x_{(\omega)}$, we have $c\leqslant x_{(\omega)}$. Since $c^2$ is compact (by the weak Kaplansky condition) the inequality $c^2\leqslant x_{(\omega)}=x_{(0)}\vee x_{(1)}\vee x_{(2)}\vee\ldots$ (where all indices are natural numbers) implies $c^2\leqslant x_{(n)}$ for some natural $n$. Hence $c\leqslant x_{(n+1)}\leqslant x_{(\omega)}$, as desired.
\end{proof}
\begin{teo}
	We have:
	\begin{itemize}
		\item [(a)] $x^{(n)}\leqslant y\Rightarrow x\leqslant y_{(n)}$ for every $x,y\in L$ and every natural $n$; in particular, $\mathrm{solv}(y)\leqslant\mathrm{Solv}(y)$.
		\item [(b)] Suppose $L$ algebraic and distributive, and satisfies the weak Kaplansky condition. Then
		\begin{equation*}
		\mathrm{Solv}(y)=\mathrm{solv}(y)=\mathrm{loc.solv}(y)=\sqrt{y}.
		\end{equation*}		 
	\end{itemize}
\end{teo}
\begin{proof}
	(a): The implication is trivial for $n=0$. If it holds for $n$, then:
	\begin{equation*}
	x^{(n+1)}\leqslant y\Leftrightarrow(x^2)^{(n)}\leqslant y\Rightarrow x^2\leqslant y_{(n)}\Rightarrow x\leqslant y_{(n+1)},
	\end{equation*}
	and so it holds for $n+1$.
	
	(b): The last two equalities hold by Theorems 6.13(c) and 6.17, respectively. After that the first equality follows from (a), since $\mathrm{Solv}(y)\leqslant\sqrt{y}$ (by Proposition 7.1).   
\end{proof}
\begin{teo}
	Under the assumptions of Theorem 7.3, the closure operator $\mathrm{Solv}:L\to L$ is algebraic.
\end{teo}
\begin{proof}
	For each $x\in L$, we have
	\begin{equation*}
	x_{(0)}=x,\,\,\,\,\,x_{(n+1)}=(x_{(n)})_{(1)}\,(\text{for each natural}\,n),\,\,\,\,\,\mathrm{Solv}(x)=x_{(0)}\vee x_{(1)}\vee x_{(2)}\vee\ldots
	\end{equation*}
	(the join over all finite ordinals = natural numbers; this follows from Theorem 7.3), and so it suffices to prove that the map $L\to L$ defined by $x\mapsto x_{(1)}$ preserves directed joins. To prove that is to prove the inequality $(\bigvee S)_{(1)}\leqslant\bigvee_{s\in S}s_{(1)}$, which, according to Lemma 7.2, is the same as
	\begin{equation*}
	c^2\leqslant\bigvee S\Rightarrow c\leqslant\bigvee_{s\in S}s_{(1)}
	\end{equation*} 
	for every compact $c\in L$ and every directed subset $S$ of $L$. Since $c^2$ is compact (by the weak Kaplansky condition) and $S$ is directed, $c^2\leqslant\bigvee S$ implies the existence of $s\in S$ with $c^2\leqslant s$, and then $c\leqslant s_{(1)}$, by definition of $s_{(1)}$.
\end{proof}
From Theorems 7.4 and 7.5, we obtain:
\begin{cor}
	Under the assumptions of Theorem 7.4(b), $L$ has algebraic radicals.\qed 
\end{cor}
And, putting this together with Theorem 4.4, we also obtain:
\begin{teo}
	$\mathrm{Spec}(L)$ is a spectral space whenever the following conditions hold:
	\begin{itemize}
		\item [(a)] $L$ is compact, algebraic, and distributive;
		\item [(b)] $L$ satisfies the weak Kaplansky condition;
		\item [(c)] if $x$ and $y$ are compact elements in $L$, then there exists a compact $c\in L$ with $\sqrt{c}=\sqrt{xy}$.\qed
	\end{itemize}
\end{teo} 
\begin{defi}\emph{
	We will say that $L$ satisfies the \textit{Kaplansky condition} if, whenever $x,y\in L$ are compact, so is $xy$.}
\end{defi}
The term \textit{Kaplansky condition} is suggested by Kaplansky's definition of a neo-commutative ring (see Fact 1.5); this also explains why we used the term \textit{weak Kaplansky condition} before. Since the Kaplansky condition implies conditions 7.7(b) and 7.7(c), Theorem 7.7 gives:
\begin{cor}
	$\mathrm{Spec}(L)$ is a spectral space whenever the following conditions hold:
	\begin{itemize}
		\item [(a)] $L$ is compact, algebraic, and distributive;
		\item [(b)] $L$ satisfies the Kaplansly condition.\qed
	\end{itemize}
\end{cor}

\section{Internal pseudogroupoids}

Let $\mathcal{C}$ be a category with finite limits. Given a span $S=$
\begin{equation*}
\xymatrix{S_0&S_1\ar[l]_-{\pi}\ar[r]^-{\pi'}&S_0'}
\end{equation*}
in $\mathcal{C}$, let us write
\begin{equation*}
\xymatrix{S_1\ar[r]^{\pi'}\ar[d]_{\pi}&S_0'&S_1\ar[l]_{\pi'}\ar[d]^{\pi}\\S_0&S_4\ar@{.>}[ul]|{\pi_1}\ar@{.>}[ur]|{\pi_3}\ar@{.>}[dl]|{\pi_2}\ar@{.>}[dr]|{\pi_4}&S_0\\S_1\ar[u]^{\pi}\ar[r]_{\pi'}&S_0'&S_1\ar[l]^{\pi'}\ar[u]_{\pi}}
\end{equation*}
for the commutative diagram whose dotted arrows form the limiting cone over the diagram formed by the solid arrows. Accordingly, for any object $X$ in $\mathcal{C}$, morphisms $f:X\to S_4$ can be displayed as $\langle f_1,f_2,f_3,f_4\rangle$, where $f_1$, $f_2$, $f_3$, and $f_4$ are morphisms from $X$ to $S_1$ making the diagram
\begin{equation*}
\xymatrix{S_1\ar[r]^{\pi'}\ar[d]_{\pi}&S_0'&S_1\ar[l]_{\pi'}\ar[d]^{\pi}\\S_0&X\ar@{.>}[ul]|{f_1}\ar@{.>}[ur]|{f_3}\ar@{.>}[dl]|{f_2}\ar@{.>}[dr]|{f_4}&S_0\\S_1\ar[u]^{\pi}\ar[r]_{\pi'}&S_0'&S_1\ar[l]^{\pi'}\ar[u]_{\pi}}
\end{equation*}
commute.
\begin{defi}\emph{
	An \textit{internal pseudogroupoid} in $\mathcal{C}$ is a pair $(S,m)$ in which $S$ is a span and $m:S_4\to S_1$ a morphism satisfying the following conditions for every morphism of the form $\langle f_1,f_2,f_3,f_4\rangle:X\to S_4$:
	\begin{itemize}
		\item [(a)] $\pi m\langle f_1,f_2,f_3,f_4\rangle=\pi f_3$ and $\pi'm\langle f_1,f_2,f_3,f_4\rangle=\pi'f_3$;
		\item [(b)] $m\langle f_1,f_2,f_3,f_4\rangle$ does not depend on $f_3$ in the sense that 
		\begin{equation*}
		m\langle f_1,f_2,f_3,f_4\rangle=m\langle f_1,f_2,f_3',f_4\rangle,
		\end{equation*}
		whenever the right-hand side of this equality makes sense;
		\item [(c)] $f_1=f_2\Rightarrow m\langle f_1,f_2,f_3,f_4\rangle=f_4$; 
		\item [(d)] $f_2=f_4\Rightarrow m\langle f_1,f_2,f_3,f_4\rangle=f_1$;
		\item [(e)] $m\langle m\langle f_1,f_2,f_3,f_4\rangle,f_5,f_3',f_6\rangle=m\langle f_1,f_2,f_3',m\langle f_4,f_5,f_3'',f_6\rangle\rangle$, whenever both sides of this equality make sense.
	\end{itemize}} 
\end{defi}
\begin{remark}\emph{
	Definition 8.1 is the same Definition 3.2 of \cite{[JP2001]}, except that \cite{[JP2001]} begins with the category $\mathsf{Set}$ of sets, and then defines an internal pseudogroupoid in $\mathcal{C}$ via the Yoneda embedding.}
\end{remark}

\section{Commutators in general categories}

Let $\mathcal{C}$ be a well-powered finitely well complete category, which means that $\mathcal{C}$ satisfies the following conditions:
\begin{itemize}
	\item $\mathcal{C}$ has finite limits;
	\item $\mathcal{C}$ is well-powered, that is, for every object $C$ in it, the class $\mathrm{Sub}(C)$ of isomorphism-classes $[X,x]$ of pairs $(X,x)$, where $x:X\to C$ is a monomorphism, is a set;
	\item each $\mathrm{Sub}(C)$, considered as an ordered set, is a complete lattice with meets being the limits of suitable diagrams in $\mathcal{C}$.
\end{itemize}
In particular, for every object $A$ in $\mathcal{C}$, we have the complete lattice $\mathrm{ER}(A)$ of (isomorphism classes of) internal equivalence relations on $A$.

The assignment $A\mapsto\mathrm{ER}(A)$ determines a functor
\begin{equation*}
\mathrm{ER}:\mathcal{C}^{\mathrm{op}}\to\mathsf{CompLat}_{\wedge},
\end{equation*}
where $\mathsf{CompLat}_{\wedge}$ is the category of complete lattices and arbitrary-meet-preserving maps. Under this functor, for a morphism $\alpha:A\to B$ in $\mathcal{C}$, the induced map $\mathrm{ER}(\alpha):\mathrm{ER}(B)\to\mathrm{ER}(A)$ is defined by pulling back along $\alpha\times \alpha:A\times A\to B\times B$.
\begin{defi}\emph{
	A \textit{commutator} $\mathrm{C}$ on $\mathcal{C}$ is a (large) family of binary operations 
	\begin{equation*}
	\mathrm{C}_A:\mathrm{ER}(A)\times\mathrm{ER}(A)\to\mathrm{ER}(A),
	\end{equation*}
	defined for each object $A$ of $\mathcal{C}$, written as $\mathrm{C}_A(x,y)=xy$ and satisfying the following conditions:
	\begin{itemize}
		\item [(a)] $xy\leqslant x\wedge y$ for all $A$ and all $x,y\in\mathrm{ER}(A)$\footnote{Here and below, $xy$ should not be confused with the composite of $x$ and $y$ as relations $A\to A$.};
		\item [(b)] $(\mathrm{ER}(\alpha)(x))(\mathrm{ER}(\alpha)(y))\leqslant \mathrm{ER}(\alpha)(xy)$ for all $\alpha:A\to B$ and all $x,y\in\mathrm{ER}(B)$.
	\end{itemize}} 
\end{defi}

Definition 9.1 is a simplified version of Definition 4.4 of \cite{[J2008]}, where a more sophisticated \textit{Galois structure with commutators} is introduced. 

From Lemma 2.6 we immediately obtain:
\begin{teo}
	For every commutator $\mathrm{C}$ on $\mathcal{C}$, every object $A$ in $\mathcal{C}$, and $\mathrm{ER}(A)$ equipped with the multiplication $\mathrm{C}_A$ (making it a complete multiplicative lattice) the space $\mathrm{Spec}(\mathrm{ER}(A))$ is sober. 
\end{teo}
Note that condition 9.1(b) plays no role here and Theorem 9.2 would still be correct without requiring it.

Let us now establish a connection with the commutators in the sense of \cite{[JP2001]}. Recall that a morphism $\varphi:S\to T$ in the category $\mathsf{Span}(\mathcal{C})$ of spans in $\mathcal{C}$ is defined as a triple $\varphi=(\varphi_0,\varphi_1,\varphi_0')$ of morphisms, making the diagram
\begin{equation*}
\xymatrix{S_0\ar[d]_{\varphi_0}&S_1\ar[l]\ar[d]|{\varphi_1}\ar[r]&S_0'\ar[d]^{\varphi_0'}\\T_0&T_1\ar[l]\ar[r]&T_0'}
\end{equation*}
whose two rows display $S$ and $T$, respectively, commute. The commutator of a span is defined as follows:
\begin{defi}\emph{
	(Definition 5.2 of \cite{[JP2001]}) Let $S$ be a span and $\Phi_S$ the class of all span morphisms $\varphi:S\to T$, in which $T$ is a span that admits an internal pseudogroupoid structure. For $(\varphi:S\to T)\in\Phi_S$, let $\mathrm{Eq}(\varphi_1)$ be the element of $\mathrm{ER}(A)$ corresponding to the kernel pair of $\varphi_1$. The \textit{commutator} $\mathrm{C}(S)$ of $S$ is defined as
	\begin{equation*}
	\mathrm{C}(S)=\bigwedge_{\varphi\in\Phi_S}\mathrm{Eq}(\varphi_1)\in\mathrm{ER}(A).
	\end{equation*}}
\end{defi}
When $\mathcal{C}$ is Barr exact, and $x=[U,u]$ and $y=[V,v]$ are elements of $\mathrm{ER}(A)$, the commutator $xy$ is defined as 
\begin{equation*}
xy=\mathrm{C}(S)\in\mathrm{ER}(A),	
\end{equation*}
where $S$ is the span
\begin{equation*}
\xymatrix{A/u&A\ar[l]\ar[r]&A/v}
\end{equation*}
in which $A\to A/u$ and $A\to A/v$ are the coequalizers of $(U,u)$ and $(V,v)$, respectively.

If $\mathcal{C}$ is not Barr exact, we have to modify this construction as follows:
\begin{itemize}
	\item Use (any) existing limit-preserving full embedding $Y:\mathcal{C}\to\mathcal{\hat{C}}$ with Barr exact $\mathcal{\hat{C}}$ satisfying the same conditions as we required for $\mathcal{C}$; for example it can be the Yoneda embedding.
	\item Call an internal pseudogroupoid $(S,m)$ in $\mathcal{\hat{C}}$ \textit{almost representable} if $S_1$ belongs to the (replete) image of $Y$. 
	\item Given $x=[U,u]$ and $y=[V,v]$ in $\mathrm{ER}(A)$, write $Y(x)$ and $Y(y)$ for the corresponding elements of $\mathrm{ER}(Y(A))$, and form the span $S=$
	\begin{equation*}
	\xymatrix{Y(A)/Y(x)&Y(A)\ar[l]\ar[r]&Y(A)/Y(y)}
	\end{equation*}
	in $\mathcal{\hat{C}}$.
	\item Consider the commutator $\mathrm{C}(S)\in\mathrm{ER}(A)$. Since $Y$ preserves all existing limits and is fully faithful, there exist a unique object in $\mathrm{ER}(A)$ corresponding via $Y$ to $\mathrm{C}(S)$, and we define $xy$ as that unique object. That is, $xy$ is defined by $Y(xy)=\mathrm{C}(S)$, in the notation above.
	\item Then, using the same arguments as in the proof of Proposition 5.4(c) of \cite{[JP2001]}, we obtain $xy\leq x\wedge y$ for all $x,y\in \mathrm{ER}(A)$. This makes $\mathrm{ER}(A)$ a complete multiplicative lattice for each object $A$ in $\mathcal{C}$, which is a special case of the complete multiplicative lattice of Theorem 9.2. In particular, commutators in the sense of \cite{[JP2001]} are special cases of commutators in the sense of Definition 9.1.	 
\end{itemize}

\section{Commutators of congruences of universal algebras}

In this section we assume that $\mathcal{C}$ is a variety of universal algebras; in particular it is Barr exact. As follows from what we explained in the previous section, for $A$ in $\mathcal{C}$, the complete multiplicative lattice $\mathrm{ER}(A)$ can be identified with the complete multiplicative lattice $\mathrm{Cong}(A)$ of all congruences on $A$, whose multiplication is defined as the commutator (operation) in the sense of \cite{[JP2001]}. The following two important special cases are mentioned in \cite{[JP2001]}: 
\begin{itemize}
   	\item [(a)] If $\mathcal{C}$ is a Mal'tsev variety (=congruence permutable variety) of universal algebras, then the commutator we use is the same as the Smith commutator \cite{[S1976]}, which was the original definition of commutator of two congruences that motivated further developments.
   	\item [(b)] More generally, if $\mathcal{C}$ is a congruence modular variety of universal algebras, then the commutator coincides with the modular commutator studied e.g. in \cite{[FM1987]} (see also references there, especially \cite{[HH1979]}, \cite{[H1979]}, \cite{[G1980]}, and \cite{[G1983]}). 
\end{itemize}

As follows from the explanation at the end of the previous section, Theorem 9.2 applies here, and so $\mathrm{Spec}(\mathrm{Cong}(A))$ is a sober space for each algebra $A$ in any variety of universal algebras.

Is $\mathrm{Spec}(\mathrm{Cong}(A))$ a spectral space? To find reasonable sufficient conditions for that we have to analyze conditions required in Theorems 3.6, 4.4, and 7.7, and in Corollary 7.9 in the case $L=\mathrm{Cong}(A)$ we are considering here. We will consider now only conditions required in Corollary 7.9:
\begin{remark}\emph{
	Let $L$ be the complete multiplicative lattice $\mathrm{Cong}(A)$. Then:
	\begin{itemize}
		\item [(a)] $x\in\mathrm{Cong}(A)$ is compact if and only if it is finitely generated as a congruence on $A$; in particular, $L$ is compact if and only if $A\times A$ is finitely generated as a congruence on $A$.
		\item [(b)] $\mathrm{Cong}(A)$ is always algebraic.
		\item [(c)] $L$ is distributive whenever $\mathcal{C}$ is congruence modular (see e.g. Proposition 4.3 in \cite{[FM1987]}). Note that in contrast to this, the monotonicity condition always holds, obviously.
		\item [(d)] $L$ satisfies the Kaplansky condition if and only if, whenever $x$ and $y$ are finitely generated congruences on $A$, so is $xy$.  
	\end{itemize}
	Here $(\mathrm{a})$ and $(\mathrm{b})$ are well known and easy to check directly, and $(\mathrm{d})$ immediately follows from $(\mathrm{a})$.} 
\end{remark}
This gives:
\begin{teo}
	The space $\mathrm{Spec}(\mathrm{Cong}(A))$ is spectral whenever the following conditions hold:
	\begin{itemize}
		\item [(a)] there is a congruence modular variety to which $A$ belongs;
		\item [(b)] $A\times A$ is finitely generated as a congruence on $A$;
		\item [(c)] if congruences $x$ and $y$ on $A$ are finitely generated, then so is their commutator $xy$.\qed
	\end{itemize}
\end{teo}
\begin{remark}\emph{
	It is very easy to find many examples where conditions 10.2(a) and 10.2(b) hold, but 10.2(c) is a very heavy condition. For example it holds for all (not necessarily unital) commutative rings, but not for all rings (see Fact 1.5), not for all groups (see e.g. \cite{[AR2005]}), and not for all Lie algebras (see e.g. \cite{[A2006]}).} 
\end{remark}
\begin{remark}\emph{
	Once the space $\mathrm{Spec}(\mathrm{Cong}(A))$ is considered, it is interesting to know of course how large it is, and, in particular, whether or not $\mathrm{Cong}(A)$ has enough prime elements. Theorem 5.7 answers the last part of this question, and, in particular tells us that:
	\begin{itemize}
		\item [(a)] When $A$ is a group, finitely generated as its normal subgroup, $\mathrm{Cong}(A)$ has enough prime elements if and only if $A$ is a perfect group. 
		\item [(b)] When $A$ is a ring, finitely generated as its ideal, $\mathrm{Cong}(A)$ has enough prime elements if and only if $A\cdot A=A$. In particular, $\mathrm{Cong}(A)$ has enough prime elements whenever $A$ is unital ring (which is always generated by its identity element).    
	\end{itemize}}
\end{remark}
\begin{remark}\emph{
	Theorem 10.2 should be compared with the following two theorems, which also say $\mathrm{Spec}(\mathrm{Cong}(A))$ is spectral, but under stronger conditions:
	\begin{itemize}
		\item [(a)] Theorem 2.9 of \cite{[A1989]}, whose requirements (although it speaks about ideals instead of congruences) can be reformulated as all requirements of our Theorem 10.2 together with: (i) $A$ belongs to an ideal determined variety; (ii) every maximal congruence in $A$ is prime.
		\item [(b)] Theorem 3.3 of \cite{[A1993]}, where $\mathrm{Cong}(A)$ is required to satisfy the ascending chain condition. This immediately implies 10.2(b) and 10.2(c), while 10.2(a) is required in \cite{[A1993]} from the beginning.
	\end{itemize}} 
\end{remark}

\section{Ideals in an abstract commutative world}

In this section $\mathcal{V}$ denotes a variety of universal algebras, and $\mathcal{C}$ the variety of universal algebras obtained from $\mathcal{V}$ by adding a commutative semigroup structure whose operation is written as $\cdot$ and such that, for all $A$ in $\mathcal{C}$ and all $a\in A$, the map $a\cdot(-):A\to A$ is a morphism in $\mathcal{V}$. For subsets $S$ and $T$ of $A$, we will write
\begin{equation*}
S\cdot T=\{s\cdot t\,|\,s\in S,\,t\in T\},
\end{equation*}
while the $\mathcal{V}$-subalgebra of $A$ generated by $S$ will be denoted by $\langle S\rangle$. We will also use the letters $v$ and $w$ to denote terms in the algebraic theory of $\mathcal{V}$ of suitable arities.
\begin{defi}\emph{
	For $A$ in $\mathcal{C}$, an \textit{ideal} of $A$ is a subalgebra $x$ of $A$ with 
	\begin{equation*}
	(a\in A\,\,\,\text{and}\,\,\,s\in x)\Rightarrow a\cdot s\in x.
	\end{equation*}}	 
\end{defi}
\begin{lem}\emph{
	The ideal of $A$ generated by a subset $S$ of $A$ is $\langle A\cdot S\rangle$.}
\end{lem}
\begin{proof}
	We only need to show that $a\cdot\langle A\cdot S\rangle\subseteq\langle A\cdot S\rangle$. Indeed, every element of $\langle A\cdot S\rangle$ can be presented in the form $v(a_1\cdot s_1,\ldots,a_n\cdot s_n)$ with $a_1,\ldots,a_n\in A$ and $s_1,\ldots,s_n\in S$, and we have
	\begin{equation*}
	a\cdot v(a_1\cdot s_1,\ldots,a_n\cdot s_n)=v(a\cdot a_1\cdot s_1,\ldots,a\cdot a_n\cdot s_n)\in\langle A\cdot S\rangle,
	\end{equation*}
	simply because $a\cdot a_1,\ldots,a\cdot a_n$ belong to $A$.
\end{proof}
The set $\mathrm{Id}(A)$ of all ideals of $A$ obviously forms a complete lattice whose meets are ordinary intersections of subalgebras, and (as easily follows from Lemma 11.2) whose joins are joins in the lattice of $\mathcal{V}$-subalgebras of $A$. Moreover, it is a complete multiplicative lattice with the multiplication defined by $xy=\langle x\cdot y\rangle$. The fact that the so defined $xy$ is indeed an ideal easily follows from Lemma 11.2. And the inequality $xy\leqslant x\wedge y$ is also straightforward.
\begin{prop}
	$\mathrm{Id}(A)$ is distributive.
\end{prop}
\begin{proof}
	Since $\mathrm{Id}(A)$ (obviously) satisfies the monotonicity condition and its multiplication is commutative, we only need to prove that $x(y\vee z)\leqslant xy\vee xz$ for all $x,y,z\in\mathrm{Id}(A)$. To prove that inequality is to prove that $x\cdot(y\vee z)\leqslant xy\vee xz$. Every element of $x\cdot(y\vee z)$ can be presented in the form $s\cdot v(t_1,\ldots,t_m,u_1,\ldots,u_n)$, and, since $s\cdot v(t_1,\ldots,t_m,u_1,\ldots,u_n)=v(s\cdot t_1,\ldots,s\cdot t_m,s\cdot u_1,\ldots,s\cdot u_n)$, it belongs to $xy\vee xz$.
\end{proof}
\begin{prop}
	If $x$ and $y$ are the ideals generated by sets $S$ and $T$, respectively, then $xy$ is the ideal generated by $S\cdot T$.  
\end{prop}
\begin{proof}
	It suffices to prove that $x\cdot y$ is a subset of the ideal generated by $S\cdot T$. By Lemma 11.2, any element of $x\cdot y$ can be written as 
	\begin{equation*}
	v(a_1\cdot s_1,\ldots,a_m\cdot s_m)\cdot w(b_1\cdot t_1,\ldots,b_n\cdot t_n),
	\end{equation*}
	where $s_1\ldots,s_m\in S$ and $t_1\ldots,t_n\in T$, and we have
	\begin{equation*}
	v(a_1\cdot s_1,\ldots,a_m\cdot s_m)\cdot w(b_1\cdot t_1,\ldots,b_n\cdot t_n),
	\end{equation*}
	\begin{equation*}
	=v(a_1\cdot s_1\cdot w(b_1\cdot t_1,\ldots,b_n\cdot t_n),\ldots,a_m\cdot s_m\cdot w(b_1\cdot t_1,\ldots,b_n\cdot t_n))
	\end{equation*}
	\begin{equation*}
	=v(w(a_1\cdot s_1\cdot b_1\cdot t_1,\ldots,a_1\cdot s_1\cdot b_n\cdot t_n),\ldots,w(a_m\cdot s_m\cdot b_1\cdot t_1,\ldots,a_m\cdot s_m\cdot b_n\cdot t_n))
	\end{equation*}
	\begin{equation*}
	=v(w(a_1\cdot b_1\cdot s_1\cdot t_1,\ldots,a_1\cdot b_n\cdot s_1\cdot t_n),\ldots,w(a_m\cdot b_1\cdot s_m\cdot t_1,\ldots,a_m\cdot b_n\cdot s_m\cdot t_n))
	\end{equation*}
	and this element belongs to $\langle A\cdot(S\cdot T)\rangle$.
\end{proof}
Similarly to Theorem 10.2, from these results and Corollary 7.9, we obtain:
\begin{teo}
	The space $\mathrm{Spec}(\mathrm{Id}(A))$ is spectral whenever $A$ is a finitely generated ideal of itself.\qed
\end{teo}
\begin{cor}
	The space $\mathrm{Spec}(\mathrm{Id}(A))$ is spectral whenever the multiplication of $A$ makes it a monoid.\qed
\end{cor}
\begin{example}\emph{
	At least in the following cases the assertion of Corollary 11.6 is known:
	\begin{itemize}
		\item [(a)] If $\mathcal{V}$ is the variety of abelian groups, then $\mathcal{C}$ is the variety of (not necessarily unital) commutative rings. In this special case Corollary 11.6 becomes a classical result of commutative algebra, which is mentioned e.g. in \cite{[H1969]} as a well-known one. Note, however, that Theorem 11.5 is also well known in this case: for example, it follows from Theorem 2.9 of \cite{[A1989]} (cf. Remark 10.5(a)), and it follows from Theorem 7 of \cite{[AR2019]}. Generally speaking, there are many similar cases where our Theorems 10.2 and 11.5 are both applicable.
		\item [(b)] If $\mathcal{V}$ is the variety of commutative monoids, then $\mathcal{C}$ is the variety of (not necessarily unital) commutative semirings. In this special case Corollary 11.6 becomes Theorem 3.1 of \cite{[PRV2009]}\footnote{\cite{[PRV2009]} requires all semirings to be unital and to have $0\neq1$, which we don't. However, this does not contradict our previous sentence, except that Theorem 3.1 of \cite{[PRV2009]} does not apply to the trivial semiring.}.
		\item [(c)] If $\mathcal{V}$ is the category of sets (considered as a variety of universal algebras), then $\mathcal{C}$ is the variety of commutative semigroups. In this special case Corollary 11.6 is known: see 4.4 in \cite{[V2010]}, where, however the term ``Kato spectrum'' is used instead of ``prime spectrum''. 
	\end{itemize}}
\end{example}

\section{Additional remarks}

\textbf{12.1. Changing morphisms of complete multiplicative lattices.} We called an adjunction $(f,u):X\to Y$ compatible if $f(x)f(x')\leqslant f(xx')$ for all $x,x'\in X$, and then called strictly compatible if this inequality is actually an equality. But what about requiring the opposite inequality $f(xx')\leqslant f(x)f(x')$, and calling this property ``op-compatibility'' (by analogy with the opmonoidal functors between monoidal categories in contrast to the monoidal ones)? Note that:
\begin{itemize}
	\item [(a)] Unlike compatibility, once we require the monotonicity condition, the op-compatibility of $(f,u):X\to Y$ can be equivalently reformulated in terms of $u$ alone: it holds if and only if $u(y)u(y')\leqslant u(yy')$ for all $y,y'\in Y$. Note also, that the counterpart of $f(1)=1$, which is the same as $1\leqslant f(1)$, becomes here the trivial condition $f(1)\leqslant1$.
	\item [(b)] The reformulation above agrees with condition (b) of Definition 9.1. This tells us that if $\mathcal{C}$ and $\mathrm{C}$ are as in Definition 9.1, then we have a commutative diagram
	\begin{equation*}
	\xymatrix{&&(\mathsf{CML}_\mathrm{s})^\mathrm{op}\ar[dl]\ar[dr]\\\mathcal{C}^\mathrm{op}\ar[drr]_-{\mathrm{ER}}\ar[r]\ar@{.>}[urr]&(\mathsf{CML}_\mathrm{op})^\mathrm{op}\ar[dr]^{(f,u)\mapsto u}&&\mathsf{CML}^\mathrm{op}\ar[dl]_{(f,u)\mapsto u}\ar[r]^{\mathrm{Spec}}&\mathsf{STop},\\&&\mathsf{CompLat}_{\wedge}}
	\end{equation*}
	in which:
	\begin{itemize}
		
		\item [($\mathrm{b}_1$)] $\mathsf{CML}_\mathrm{op}$ is the same as $\mathsf{CML}$, except that its morphisms are required to be op-compatible adjunctions.
		\item [($\mathrm{b}_2$)] The horizontal unlabeled arrow denotes the functor opposite to the functor $\mathrm{C}^*:\mathcal{C}\to\mathsf{CML}_\mathrm{op}$ defined by $\mathrm{C}^*(\alpha:A\to B)=(f,u):\mathrm{ER}(A)\to\mathrm{ER}(B)$, where $u=\mathrm{ER}(\alpha)$ and so $f$ is the left adjoint of $\mathrm{ER}(\alpha)$. In the situation of Section 10, after identifying $\mathrm{ER}(A)$ with the lattice of congruences on $A$, we can say that, for $x\in\mathrm{ER}(A)$, $f(x)$ is defined as the congruence on $B$ generated by the image of $x$ under the homomorphism $\alpha\times\alpha:A\times A\to B\times B$.
		\item [($\mathrm{b}_3$)] The other unlabeled solid arrows denote the inclusion functors.
		\item [($\mathrm{b}_4$)] The dotted arrow, uniquely determined by commutativity of our diagram, when it exists, in fact exists only in some special cases. For example, in the situation of Section 10, it exists when $\mathcal{C}$ is the category of commutative rings, but not  when $\mathcal{C}$ is the category of (all) rings.
		\item [($\mathrm{b}_5$)] If, however, again in the situation of Section 10, we take $\mathcal{C}$ to be any congruence modular variety, but restrict its morphisms, namely take only surjections, then the dotted arrow always exists. That is, in the notation of ($\mathrm{b}_2$), if $\mathcal{C}$ is a congruence modular variety of universal algebras and $\alpha:A\to B$ is surjective, then $f(xx')=f(x)f(x')$ for all $x,x'\in\mathrm{ER}(A)$. In the notation of \cite{[JP2001]}, this equality would be written as $\alpha_\#[x,x']=[\alpha_\#x,\alpha_\#x']$ (see Theorem 8.1 there, although it uses different letters). This well-known fact goes back to first papers on `modular commutators' and plays an important role in commurator theory (see e.g. \cite{[FM1987]}). 
	\end{itemize}
\end{itemize} 

\textbf{12.2. Commutative reflections.} Let $\mathsf{MCML}$ and $\mathsf{CMCML}$ be the full subcategories of $\mathsf{CML}$ with objects all objects of $\mathsf{CML}$ that satisfy the monotonicity condition, and that satisfy the monotonicity condition and also have commutative multiplication, respectively. Let $\mathsf{MCML}_\mathrm{s}$ and $\mathsf{CMCML}_\mathrm{s}$ be similar full subcategories of $\mathsf{CML}_\mathrm{s}$, and $\mathsf{MCML}_\mathrm{op}$ and $\mathsf{CMCML}_\mathrm{op}$ be similar full subcategories of $\mathsf{CML}_\mathrm{op}$. Going back to Example 3.7, let us write informally $(1_L,1_L):\mathrm{Com}(L)\to L$ for the compatible adjunction considered there. By definition of $\mathsf{CML}$, this makes $(1_L,1_L)$ a morphism from $\mathrm{Com}(L)$ to $L$ in $\mathsf{CML}$. On the other hand, the same pair $(1_L,1_L)$ can also be considered as a morphism from $L$ to $\mathrm{Com}(L)$ in $\mathsf{CML}_\mathrm{op}$. Moreover, we have:
\begin{itemize}
	\item [(a)] The assignment $X\mapsto\mathrm{Com}(X)$ together with all compatible adjunctions of the form $(1_X,1_X):\mathrm{Com}(X)\to X$ determines a functor 
	\begin{equation*}
	\mathsf{MCML}\to \mathsf{CMCML},
	\end{equation*}
	which is the right adjoint of the inclusion functor. To prove all this, just note that, for every compatible adjunction $(f,u):X\to Y$ with $X$ in $\mathsf{CMCML}$ and $Y$ in $\mathsf{MCML}$, we have
	\begin{equation*}
	f(x)*f(x')=f(x)f(x')\vee f(x')f(x)\leqslant f(xx')\vee f(x'x)
	\end{equation*}
	\begin{equation*}
	\leqslant f(xx')\vee f(xx')=f(xx')
	\end{equation*}
	for all $x,x'\in X$, making $(f,u):X\to\mathrm{Com}(Y)$ a compatible adjunction.
	\item [(b)] The assignment $X\mapsto\mathrm{Com}(X)$ together with all compatible adjunctions of the form $(1_X,1_X):X\to\mathrm{Com}(X)$ determines a functor 
	\begin{equation*}
	\mathsf{MCML}_\mathrm{op}\to \mathsf{CMCML}_\mathrm{op},
	\end{equation*}
	which is the left adjoint of the inclusion functor. To prove all this, just note that, for every compatible adjunction $(f,u):X\to Y$ with $X$ in $\mathsf{MCML}_\mathrm{op}$ and $Y$ in $\mathsf{CMCML}_\mathrm{op}$, we have
	\begin{equation*}
	f(x*x')=f(xx'\vee x'x)=f(xx')\vee f(x'x)
	\end{equation*}
	\begin{equation*}
	\leqslant f(x)f(x')\vee f(x')f(x)=f(x)f(x')\vee f(x)f(x')=f(x)f(x')
	\end{equation*}
	for all $x,x'\in X$, making $(f,u):\mathrm{Com}(X)\to Y$ a compatible adjunction.
\end{itemize}
In particular, $\mathsf{CMCML}$ is a coreflective full subcategory of $\mathsf{MCML}$, while $\mathsf{CMCML}_\mathrm{op}$ is a reflective full subcategory of $\mathsf{MCML}_\mathrm{op}$. What about $\mathsf{CMCML}_\mathrm{s}$ and $\mathsf{MCML}_\mathrm{s}$? Well, independently of the monotonicity condition it is easy to see that:
\begin{itemize}
	\item [(c)] $\mathsf{CMCML}_\mathrm{s}$ is a reflective full subcategory of $\mathsf{MCML}_\mathrm{s}$, but not with a reflection that has $X\to\mathrm{Com}(X)$.
	\item [(d)] $\mathsf{CMCML}_\mathrm{s}$ is not coreflective in $\mathsf{MCML}_\mathrm{s}$.   
\end{itemize}

\textbf{12.3. The functor Spec as a right adjoint.} Now let us go back to Theorem 3.2 and what we considered between it and Theorem 3.6, and observe:
\begin{itemize}
	\item [(a)] Let $Y$ be a spatial frame (=spatial locale) considered as a complete multiplicative lattice with $yy'=y\wedge y'$ for all $y,y'\in Y$. Then $y=y'$ in $Y$ if and only if $y\leqslant p\Leftrightarrow y'\leqslant p$ for every prime element $p$ in $Y$. This is a well known fact in frame theory; in fact a frame is spatial if and only if it satisfies this condition.  
	\item [(b)] Let $(f,u):X\to Y$ be a compatible adjunction between complete multiplicative lattices. If $Y$ is a frame, then
	\begin{equation*}
	\sqrt{x}=\sqrt{x'}\Rightarrow f(x)=f(x')
	\end{equation*}
	for all $x,x'\in X$. Indeed, we have
	\begin{equation*}
	\sqrt{x}=\sqrt{x'}\Rightarrow((p\,\text{is prime in}\,Y)\Rightarrow(x\leqslant u(p)\Leftrightarrow x'\leqslant u(p)))
	\end{equation*}
	\begin{equation*}
	\Leftrightarrow((p\,\text{is prime in}\,Y)\Rightarrow(f(x)\leqslant p\Leftrightarrow f(x')\leqslant p))\Leftrightarrow f(x)=f(x'),
	\end{equation*}
	using Theorem 3.2(a) and observation (a).
	\item [(c)] As easily follows from (b), the assignment $X\mapsto\sqrt{X}$ together with all compatible adjunctions of the form $X\to\sqrt{X}$, constructed as in Remark 3.5(c), determines a functor from $\mathsf{CML}$ to the category $(\mathsf{CLoc})^\mathrm{op}$ of spatial frames, which is the left adjoint of the inclusion functor.
	\item [(d)] From (c) and known results on locales/frames recalled in Section 3, we obtain the commutative diagram
	\begin{equation*}
	\xymatrix{\mathsf{CML}^\mathrm{op}\ar[d]_{\mathrm{coreflection}}\ar[drr]|{X\mapsto\sqrt{X}}\ar[rr]^{\mathsf{Spec}}&&\mathsf{STop}\\\mathsf{Loc}\ar[rr]_{X\mapsto\sqrt{X}}&&\mathsf{SLoc}\ar[u]_{{\mathsf{Spec}|}_{\mathsf{SLoc}}}}
	\end{equation*}
	in which:
	\begin{itemize}
		\item [($\mathrm{d}_1$)] According to (c), the diagonal arrow is the coreflection from $\mathsf{CML}^\mathrm{op}$ to its full subcategory $\mathsf{SLoc}$ of spatial locales.
		\item [($\mathrm{d}_2$)] The left-hand vertical arrow is well defined since: (i) $\mathsf{Loc}^{\mathrm{op}}$ can be seen as an equationally defined full subcategory of the infinitary algebraic category $\mathsf{CML}_\mathrm{s}$; (ii) replacing $\mathsf{CML}_\mathrm{s}$ with $\mathsf{CML}$ does not change anything since every morphism in $\mathsf{CML}$, whose codomain is a frame, automatically belongs to $\mathsf{CML}_\mathrm{s}$.
		\item [($\mathrm{d}_3$)] The triangle involving $\mathsf{Loc}$ indeed commutes since all functors there are right adjoints of the inclusion functors.
		\item [($\mathrm{d}_4$)] The fact that the triangle involving $\mathsf{Stop}$ commutes easily follows our observations in Remark 3.5(c).
	\end{itemize}
	\item [(e)] As follows from (d) we could simply define the functor
	\begin{equation*}
	\mathrm{Spec}:\mathsf{CML}^{\mathrm{op}}\to\mathsf{STop}
	\end{equation*}
    as the right adjoint of the composite of
    \begin{equation*}
    \xymatrix{\mathsf{STop}\ar[rrr]^{\mathrm{standard\,equivalence}}&&&\mathsf{SLoc}\ar[rr]^{\mathrm{inclusion}}&&\mathsf{CML}^{\mathrm{op}}}
    \end{equation*}
    or, equivalently, as the composite of
    \begin{equation*}
    \xymatrix{\mathsf{STop}\ar[rr]^{\mathrm{inclusion}}&&\mathsf{Top}\ar[r]^{\Omega}&\mathsf{Loc}\ar[rr]^{\mathrm{inclusion}}&&\mathsf{CML}^{\mathrm{op}}.}
    \end{equation*}
    \item [(f)] This suggests removing the condition $xy\leqslant x\wedge y$ from our definition of complete multiplicative lattice, call such more general structure a, say, complete pseudo-multiplicative lattice, call the category of such structures $\mathsf{PCML}$, and extend the functor $\mathrm{Spec}$ to $\mathsf{PCML}$ by taking the composite
    \begin{equation*}
    \xymatrix{\mathsf{PCML}^{\mathrm{op}}\ar[rr]^{\mathrm{coreflection}}&&\mathsf{CML}^{\mathrm{op}}\ar[rr]^{\mathrm{Spec}}&&\mathsf{STop},}
    \end{equation*}
    but we don't see yet any interesting example of this construction. 
\end{itemize}

\textbf{12.4. Semiprime elements.} Let us call an element $s$ of $L$ \textit{semiprime} if it satisfies the implication 
\begin{equation*}
x^2\leqslant s\Rightarrow x\leqslant s,
\end{equation*}
or, equivalently, it has $x^2\leqslant s\Leftrightarrow x\leqslant s$, for all $x\in L$. It is clear that the set of all semiprime elements of $L$ is closed under arbitrary meets in $L$, and so we have a closure operator $\mathrm{sp}:L\to L$, which associates, to each $x\in L$, the smallest semiprime $s\in L$ with $x\leqslant s$. Let us compare $\mathrm{sp}(x)$ with $\sqrt{x}$ for an arbitrary $x\in L$:
\begin{itemize}
	\item [(a)] Since every prime element is (obviously) semiprime, so is every radical element and we have $\mathrm{sp}(x)\leqslant\sqrt{x}$.
	\item [(b)] Using transfinite induction, it is easy to see that $\mathrm{Solv}(x)\leqslant\mathrm{sp}(x)$.
	\item [(c)] Therefore whenever $\sqrt{x}\leqslant\mathrm{Solv}(x)$, we have $\mathrm{Solv}(x)=\mathrm{sp}(x)=\sqrt{x}$. In particular, these equalities hold under the assumptions of Theorem 7.4(b).
	\item [(d)] We are interested in proving the equality $\mathrm{sp}(x)=\sqrt{x}$ under weaker conditions, specifically only requiring $L$ to be distributive and algebraic, as in Theorem 6.17, which avoids the weak Kaplansky condition required in Theorem 7.4(b). Having (a) and Theorem 6.17 in mind, we only need to prove the inequality $\mathrm{loc.solv}(x)\leqslant\mathrm{sp}(x)$. This can be done by copying an argument from the proof of Proposition 6.14 replacing the prime element used there with a semiprime one. That is, we have: If $L$ is distributive and algebraic, then $\mathrm{sp}(x)=\sqrt{x}$ for every $x\in L$. However, this equality is the same as Theorem A in \cite{[K1972]}; recall: that Theorem A corrects Satz 2.5 of \cite{[S1968]} (see also \cite{[S1971]} for a more general result). 
\end{itemize}

\textbf{12.5. `Not enough' primes:} Not in the sense of Section 5, but in the sense that there are no primes strictly between $xy$ and $x\wedge y$. More precisely, for $p,x,y\in L$, if $p$ is prime, then the following conditions are equivalent:
\begin{itemize}
	\item [(a)] $xy\leqslant p\leqslant x\wedge y$;
	\item [(b)] $(x\leqslant p$ or $y\leqslant p)$, $p\leqslant x$, and $p\leqslant y$; 
	\item [(c)] $x=p\leqslant y$ or $y=p\leqslant x$;
	\item [(d)] $p=x\wedge y$.
\end{itemize}
Indeed, each implication in (a)$\Rightarrow$(b)$\Rightarrow$(c)$\Rightarrow$(d)$\Rightarrow$(a) is obvious.
\vspace{1mm}

\textbf{12.6. Involving different kinds of commutators.} Considering complete multiplicative lattices of internal equivalence relations and, in particular, congruences of universal algebras in Sections 9 and 10, we were using the following kinds of commutators:
\begin{itemize}
	\item [(a)] General commutators, only required to satisfy `most mild' axioms (see Definition 9.1 with reference to \cite{[J2008]}).
	\item [(b)] Commutators defined via pseudogroupoids \cite{[JP2001]}, which generalize the modular commutator, in fact considered `the most standard' in universal algebra (the references are given in Section 10), and, in particular, the Smith commutator \cite{[S1976]}.  
\end{itemize}
Now let us add, assuming for simplicity that all categories we are talking about satisfy the three conditions given at the beginning of Section 9 and have finite colimits:
\begin{itemize}
	\item [(c)] The notion of Smith commutator was extended from Mal'tsev varieties to Barr exact Mal'tsev categories by M. C. Pedicchio \cite{[P1995]}, but this is also a special case of the commutator in the sense of \cite{[JP2001]}. 
	\item [(d)] The commutator used in the papers mentioned in the first sentence of Fact 1.8 (of Section 1) is either Huq's commutator in categories satisfying Huq's axioms, or Higgins' commutator (for normal subobjects) in varieties of groups with multiple operators. The relationship between these two commutators was clarified by A. S. Cigoli, J. R. A. Gray, and T. Van der Linden \cite{[CGV2015]} in the context of semi-abelian categories \cite{[JMT2002]}, whose `old-style' axioms are essentially the same as Huq's axioms. See also related previous work of S. Mantovani and G. Metere \cite{[MM2010]}, who in fact introduced the categorical generalization of Higgins commutator in the more general context of ideal determined categories in the sense of \cite{[JMTU2010]}. On the other hand, the weighted commutators introduced in \cite{[GJU2014]}, again in the semi-abelian context, give the Smith\textendash Pedicchio commutator and the Huq commutator as special cases. And of course each weighted commutator produces its own complete multiplicative lattice. As follows from the results of Sections 9 and 10, this gives many examples of sober and even spectral (=coherent) spectra.
	\item [(e)] Note, however, that the Smith\textendash Pedicchio and the Huq commutator often coincide making the weighted commutator independent of its weight. This is the case when the ground semi-abelian category $\mathcal{C}$ is strongly protomodular \cite{[B2004]}, or action accessible \cite{[BJ2009]}. In particular, as follows from the action accessibility theorem of A. Montoli \cite{[M2010]}, it is the case when $\mathcal{C}$ is a `category of interest' in the sense of G. Orzech \cite{[O1972]}, which applies to many categories of classical algebraic structures. Various further relevant comments about the coincidence of two commutators can be found in \cite{[BG2002]}, \cite{[GV2008]}, \cite{[MV2012]}, \cite{[BMV2013]}, \cite{[MV2014]}, and \cite {[MV2015]}. The list of known algebraic structures whose categories have the two commutators different from each other includes loops (where, as mentioned in \cite{[GJU2014]}, this can be deduced from the result of Exercise 10 of Chapter 5 in \cite{[FM1987]}), digroups (=sets equipped with two independent group structures with the same identity element; see e.g. \cite{[B2004]}), and near-rings (as shown in \cite{[JMV2016]}; note, on the other hand, that prime ideals of near-rings were introduced and studied in \cite{[V1964]}).       
	\item [(f)] There is more to say about various non-modular commutators introduced in universal algebra, the relative commutator in the sense of T. Everaert and T. Van der Linden \cite{[EV2012]}, and commutator theory and related studies in regular (not necessarily Barr exact) categories developed in several papers of D. Bourn and M. Gran, but we omit it here.	  
\end{itemize}
\textbf{12.7. Rings.} We already mentioned the case of rings in Introduction and in Remarks 10.3 and 10.4(a). Now let us add, partly with repetitions:
\begin{itemize}
	\item [(a)] For a ring $A$, the complete multiplicative lattice $\mathrm{Cong}(A)$ considered in Section 10 and, in particular, in Theorem 10.2, can be identified of course with the complete lattice $\mathrm{Id}(A)$ of ideals of $A$, with the multiplication defined by $xy=x\cdot y+y\cdot x$, where $\cdot$ denotes the usual multiplication of ideals (this notation should not be confused with what we used in Section 11).
	\item [(b)] In particular an ideal $p$ of $A$ is defined to be prime, if, for all ideals $x$ and $y$ of $A$, we have
	\begin{equation*}
	x\cdot y+y\cdot x\leqslant p\Rightarrow (x\leqslant p\,\,\,\text{or}\,\,\,y\leqslant p).
	\end{equation*}
	However, as follows from the result of 12.2, this definition is equivalent to the usual one, which uses $x\cdot y$ instead of $x\cdot y+y\cdot x$:
	\begin{equation*}
	x\cdot y\leqslant p\Rightarrow (x\leqslant p\,\,\,\text{or}\,\,\,y\leqslant p).
	\end{equation*}
	It seems that in the case of rings this equivalence was first noticed by G. K. Gerber in \cite{[G1983]} (see also \cite{[BG1985]} for a more general observation). 
	\item [(c)] As follows from (b), $\mathrm{Spec}(\mathrm{Cong}(A))=\mathrm{Spec}(\mathrm{Id}(A))$ is the same as the classical prime spectrum of $A$.
	\item [(d)] Theorem 10.2 becomes: If (i) $A$ is finitely generated as an ideal of itself, and (ii) the product of any two finitely generated ideals of $A$ is a finitely generated ideal, then $\mathrm{Spec}(\mathrm{Cong}(A))=\mathrm{Spec}(\mathrm{Id}(A))$ is a spectral space. Here: (i) holds whenever $A$ unital (cf. 10.4(b)), and (ii) is exactly Kaplansky's definition of neo-commutativity for $A$ \cite{[K1974]}. That is, when $A$ is unital, Theorem 10.2 becomes Kaplansky's result on the spectrality of the prime spectrum of $A$. This explains what we said about Kaplansky's and Belluce's results in Fact 1.5. Specifically, we see that Theorem 10.2 generalizes Kaplansky's result, but to generalize Belluce's result we would have to prove Theorem 7.7 omitting condition 7.7(b) from its assumptions, or, equivalently, prove Corollary 7.9 assuming 7.7(c) instead of 7.9(c). We don't know whether this is possible or not.   
\end{itemize}
	
{}
\end{document}